\documentclass[12pt]{amsart} 

\usepackage{amssymb} 
\usepackage{amsmath} 
\usepackage{fullpage} 
\usepackage{url}

\newtheorem{theo}{Theorem}
\newtheorem{conj}{Conjecture}
\newtheorem{prop}{Proposition}
\newtheorem{lemma}{Lemma}
\newtheorem{cor}{Corollary}
\theoremstyle{definition}
\newtheorem{defi}{Definition}
\theoremstyle{remark}
\newtheorem{rem}{Remark}

\def\stacksum#1#2{{\stackrel{{\scriptstyle #1}}
{{\scriptstyle #2}}}}

\def\Gb{{\mathbf G}}

\def\Cr{{\mathbb C}}

\def\Er{{\mathbb E}}
\def\Fr{{\mathbb F}}

\def\Nr{{\mathbb N}}

\def\Pr{{\mathbb P}}

\def\Rr{{\mathbb R}}

\def\Tr{{\mathbb T}}

\def\Zr{{\mathbb Z}}
\def\Cc{{\mathcal{C}}}

\def\Fc{{\mathcal{F}}}

\def\one{{\rm \bf 1}}

\def\det{\operatorname{det}}
\def\({\left(}     
\def\){\right)}    
\def\[{\left[}     
\def\]{\right]}
\def\dl{\(\!\!\(}
\def\dr{\)\!\!\)}
\def\ddl{\(\!\(}
\def\ddr{\)\!\)}  
\def\ep{\varepsilon} 
\begin{document}
\title{Mod-$\varphi$ convergence} 
\author{Freddy Delbaen}
\author{Emmanuel Kowalski}
\address{ETH Z\"urich -- D-MATH \\ R\"{a}mistrasse
   101\\ 8092 Z\"{u}rich, Switzerland}
\author{Ashkan Nikeghbali}
 \address{Institut f\"ur Mathematik,
 Universit\"at Z\"urich, Winterthurerstrasse 190,
 8057 Z\"urich,
 Switzerland}
\date{\today}



\begin{abstract}
  Using Fourier analysis, we study local limit theorems in
  weak-convergence problems.  Among many applications, we discuss
  random matrix theory, some probabilistic models in number theory,
  the winding number of complex brownian motion and the classical
  situation of the central limit theorem, and a conjecture concerning
  the distribution of values of the Riemann zeta function on the
  critical line.
\end{abstract}

\maketitle

\section{Introduction}

In \cite{JKN}, the notion of \emph{mod-gaussian convergence} was
introduced: intuitively, it corresponds to a sequence of random
variables $X_n$ that -- through the Fourier lens -- ``look like'' a
sum $G_n+Y_n$ where $(G_n)$ is a sequence of gaussian variables with
arbitrary variance and $(Y_n)$ is a convergent sequence independent
from $G_n$. However, most interest lies in cases where this
simple-minded decomposition does not exist: what remains is the
existence of a limiting function $\Phi$, not necessarily a Fourier
transform of a probability measure, such that the limit theorem
\begin{equation}\label{eq-orig-def}
\lim_{n\rightarrow +\infty} \Er[e^{itG_n}]^{-1}\Er[e^{itX_n}]=
\Phi(t)
\end{equation}
holds, locally uniformly, for $t\in\Rr$.
\par
In the same spirit, we introduce in this paper a notion of
``convergence'' where the reference law is not necessarily gaussian
but a fairly general probability law, with integrable characteristic
function $\varphi$. Under suitable conditions, we are able to prove a
general local limit theorem which extends the result recently found
in~\cite[Th. 4]{KN}.
\par
As illustrations of the consequences of this framework, we mention two
results which, to the best of our knowledge, are new:

\begin{theo}[Local limit theorem for the winding number of complex
  brownian motion]
  For $u\geq 0$, let $\theta_u$ denote the argument or winding number
  of a complex brownian motion $W_u$ such that $W_0=1$. Then for any
  real numbers $a<b$, we have
$$
\lim_{u\rightarrow \infty}\frac{\log u}{2}\,
\Pr\[a<\theta_u<b\]=\frac{1}{\pi}(b-a).
$$
\end{theo}

This is proved in Section~\ref{sec-winding}.

\begin{theo}[Local limit theorem for unitary matrices]\label{th-unitary}
  For $n\geq 1$, let $g_n$ denote a random matrix which is
  Haar-distributed in the unitary group $U(n)$. Then for any bounded
  Borel subset $B\subset \Cr$ with boundary of Lebesgue measure $0$,
  and for any $b\in\Cr$,  we have
$$
\lim_{n\rightarrow +\infty} \Bigl(\frac{\log n}{2}\Bigr) {\Pr\[\log
  \det(1-g_n)-\frac{\log n}{2}\ b\in B\]}=\frac{1}{2\pi}e^{-|b|^2/2}m(B),
$$
where $m(\cdot)$ denotes the Lebesgue measure on $\Cr$.
\end{theo}

This, together with similar facts for the other families of classical
compact groups, is proved in Section~\ref{sec-random-matrices}.
\par
We emphasize that these two applications are just examples; this paper
contains quite a few more, and it seems certain that many more
interesting convergence theorems can be proved or understood using the
methods of this paper.
\par
\medskip
\par
\textbf{Notation and preliminaries.}  Our random variables will take
values in $\Rr^d$, a fixed $d-$dimensional space, and we denote by
$|t|$ the Euclidian norm in $\Rr^d$. We will use the Landau and
Vinogradov notations $f=O(g)$ and $f\ll g$ in some places; these are
equivalent statements, and mean that there exists a constant $c\geq 0$
such that
$$
|f(x)|\leq cg(x)
$$
for all $x$ in a set $X$ which is indicated. Any suitable value of $c$
is called ``an implied constant'', and it may depend on further
parameters.
\par
A sequence of probability measures $(\mu_n)_n$ on $\Rr^d$ converges
weakly to a probability measure if for all bounded continuous
functions $f\colon \Rr^d\rightarrow \Rr$, we have $\lim_n\int
f\,d\mu_n=\int f\,d\mu$.  Equivalently we can ask that the convergence
holds for $C^\infty$ functions with compact support. L\'evy's
theorem asserts that this is equivalent to the pointwise convergence
of the characteristic functions $\int \exp(i\,t\cdot
x)\,d\mu_n\rightarrow \int \exp(i\,t\cdot x)\,d\mu$.  L\'evy's theorem
can be phrased as follows.  If $\varphi_n$ is the sequence of
characteristic functions of probability measures $\mu_n$, if
$\varphi_n(t)$ converges pointwise to a function $\varphi$, if this
convergence is continuous at the point $0$, then $\varphi$ is a
characteristic function of a probability measure $\mu$, $\mu_n$
converges weakly to $\mu$ and the convergence of $\varphi_n$ to
$\varphi$ is uniform on compact sets of $\Rr^d$.  We recall that the
convergence is continuous if $x_n\rightarrow 0$ in $\Rr^d$ implies
$\varphi_n(x_n)\rightarrow \varphi(0)=1$.
\par
We say that a sequence of random variables converges in law if the
image measures (or laws) converge weakly.  Most of the time one needs
a scaling of the sequence.  This is for instance the case in the
central limit theorem, which in an elementary form says that for a
sequence of independent identically distributed real-valued random
variables, $(X_n)_n$, $\Er[X_n]=0,\Er[X_n^2]=1$, the normalised (or
rescaled) sequence $\frac{X_1+\ldots X_n}{\sqrt{n}}$ converges weakly
to the standard gaussian law.  
\par
In the applications below we will use different kinds of scaling. In
the higher-dimensional case, we will scale the random variables using
a sequence of linear isomorphisms (or non-degenerate matrices)
$A_n\colon \Rr^d\rightarrow \Rr^d$.  The inverse of these matrices
will be denoted by $\Sigma_n\colon\Rr^d\rightarrow \Rr^d$. The
transpose of a linear map or matrix $A$ is denoted by $A^*$.
\par
Our methods are based on Fourier analysis and we will use basic facts
from this theory freely. We define the Fourier transform as is usually
done in probability theory, namely
$$
\hat{f}(t)=\int_{\Rr^d} \exp(i\,t\cdot x)f(x)\,m(dx).
$$
\par
The inversion formula is, at least when $\hat{f}\in L^1(\Rr^d)$, given
by 
$$
f(x)=\(\frac{1}{2\pi}\)^d\int_{\Rr^d}\exp(-it\cdot x)\hat{f}(t)\,dt.
$$
\par
In particular, when $\mu$ is a probability measure with an integrable
characteristic function $\varphi$, we get that $\mu$ is absolutely
continuous with respect to Lebesgue measure $m$, $\mu\ll m$, and its
density is given by 
$$
\frac{d\mu}{dm}(x)=\(\frac{1}{2\pi}\)^d\int \exp(-i
  t\cdot x)\varphi(t)\,m(dt),
$$
which is therefore continuous. 
\par
The proof of our main result (Theorem~\ref{th-local-limit}) is based
on the following approximation theorem:

\begin{theo}\label{th-approx}
  Suppose $f\colon\Rr^d\rightarrow \Rr$ is a continuous function with
  compact support.  Then for each $\eta>0$ we can find two functions
  $g_1,g_2\colon\Rr^d\rightarrow \Rr$ such that
\begin{enumerate}
\item $\widehat{g_1},\widehat{g_2}$ have compact support,
\item $g_2\le f \le g_1$,
\item $\int_{\Rr^d}(g_1-g_2)(t)\,m(dt)\le \eta$.
\end{enumerate}
\end{theo} 

This is a standard result in Fourier analysis, see e.g. Bretagnolle
and Dacunha-Castelle~\cite{BDC}; for the sake of completeness, we
sketch a proof in Appendix A.

\section{Mod$-\varphi$ convergence}

\subsection{Definition}

We now explain our generalization of the definition
in~\cite{JKN}. First of all, we fix $d\geq 1$ and a probability
measure $\mu$ on $\Rr^d$. We then assume given a sequence $(X_n)$ of
random variables defined on a probability space $(\Omega,\Fc,\Pr)$ and
taking values also in $\Rr^d$. We define $\varphi_n$ to be the
characteristic function of $X_n$. We now consider the following
properties:
\par
\medskip
\begin{itemize}
\item{\bf H1.} The characteristic function $\varphi$ of the
  probability measure $\mu$ is integrable; in particular, $\mu$ has a
  density $d\mu/dm$, with respect to Lebesgue measure $m$.
\item{\bf H2.} There exists a sequence of linear automorphisms
  $A_n\in\mathrm{GL}_d(\Rr)$, with inverses $\Sigma_n=A_n^{-1}$, such
  that $\Sigma_n$ converges to $0$ and $\varphi_n(\Sigma^*_n t)$
  converges continuously at $0$ (or what is equivalent: uniformly on
  compact sets) to $\varphi(t)$.  In other words, the renormalized
  random variables $\Sigma_n(X_n)$ converge in law to $\mu$.
\item{\bf H3.} For all $k\geq 0$, the sequence
$$
f_{n,k}=\varphi_n(\Sigma_n^*t)\one_{|\Sigma_n^*t|\le k}
$$ 
is \emph{uniformly integrable} on $\Rr^d$; since $f_{n,k}$ are
uniformly bounded in $L^1$ and $L^{\infty}$ (for fixed $k$), this is
equivalent to the statement that, for all $k\geq 0$, we have
\begin{equation}\label{eq-formal-h3}
\lim_{a\rightarrow +\infty} \sup_{n\geq 1}{\int_{|t|\geq
    a}|\varphi_n(\Sigma_n^*t)| \one_{|\Sigma_n^*t|\le k}m(dt)}=0.
\end{equation}
\end{itemize}

\begin{rem}
  Property {\bf H1} excludes discrete probability laws, such as
  Poisson random variables. However, similar ideas do apply for such
  cases. We refer to~\cite{KN2} (for the case of Poisson
  distributions) and to~\cite{BKN} (for much more general discrete
  distributions, where earlier work of Hwang~\cite{Hwang} are also
  relevant) for these developments.
\end{rem}

\begin{rem}\rm 
  Property {\bf H3} will typically be established by proving an
  estimate of the type
\begin{equation}\label{eq-dominated}
|\varphi_n(\Sigma_n^*t)|\leq h(t)
\end{equation}
for all $n\geq 1$ and all $t\in \Rr^d$ such that $|\Sigma_n^*t|\leq
k$, where $h\geq 0$ is an integrable function on $\Rr^d$ (which may
depend on $k$).
\end{rem}

We give a name to sequences with these properties:

\begin{defi}[Mod-$\varphi$ convergence]\label{def-modphi}
  If $\mu$ is a probability measure on $\Rr^d$ with characteristic
  function $\varphi$, $X_n$ is a sequence of $\Rr^d$-valued random
  variables with characteristic functions $\varphi_n$, and if the
  properties {\bf H1}, {\bf H2}, {\bf H3} hold, then we say that there
  is \emph{mod$-\varphi$ convergence} for the sequence $X_n$.
\end{defi}

Below, we will comment further on the hypotheses, and in particular
give equivalent formulations of {\bf H3}. In
Section~\ref{sec-central-limit}, we also explain the relation with
conditions arising in classical convergence theorems.
\par 
To make the link with the original definition in~\cite{JKN}, i.e., the
assumption that a limit formula like~(\ref{eq-orig-def}) holds, we
observe that mod-$\varphi$ convergence will hold when {\bf H1} is true
and we have
\begin{itemize}
\item{\bf H2'.} There exists a sequence of linear automorphisms
  $A_n\in\mathrm{GL}_d(\Rr)$, with inverses $\Sigma_n=A_n^{-1}$, such
  that $\Sigma_n$ converges to $0$, and there exists a continuous
  function $\Phi\,:\, \Rr^d\rightarrow \Cr$ such that
\begin{equation}\label{eq-original}
\varphi_n(t)=\Phi(t)\varphi(A_n^*t)(1+o(1))
\end{equation}
uniformly for $t$ such that $|\Sigma_n^*t|\leq k$, for arbitrary
$k>0$.
\end{itemize}

In many applications considered in this paper (not all), this stronger
condition holds, or is expected to hold. It is very likely that, when
this is the case, the ``limiting function'' $\Phi$ also carries
significant information, as discussed already in special cases
in~\cite[\S 4]{JKN}.

\subsection{Local limit theorem}

We now state and prove our main result, which is a local limit theorem
that shows that, when mod$-\varphi$ convergence holds, the
expectations $\Er\[f(X_n)\]$ (for reasonable functions $f$) do not
converge, but are well-controlled: they behave like
$$
|\det(A_n)|^{-1}\frac{d\mu}{dm}(0)\int_{\Rr^d} f\,dm
$$
as $n$ goes to infinity. The proof turns out to be very simple:


\begin{theo}[Local limit theorem for mod-$\varphi$
  convergence]\label{th-local-limit}
  Suppose that mod$-\varphi$ convergence holds for the sequence $X_n$.
  Then we have
$$
|\det(A_n)|\,\Er\[f(X_n)\]\rightarrow \frac{d\mu}{dm}(0)\int_{\Rr^d} f\,dm,
$$
for all 
continuous functions with compact support. Consequently we also have
\begin{equation}\label{eq-limit-jordan}
|\det(A_n)|\,\Pr\[X_n\in B\]\rightarrow \frac{d\mu}{dm}(0)\,m(B).
\end{equation}
for relatively compact Borel sets $B\subset\Rr^d$ with $m(\partial
B)=0$, or in other words for bounded Jordan-measurable sets $B\subset
\Rr^d$.
\end{theo}




\begin{proof}
  We first assume that $f$ is such that $\hat{f}$ has compact support;
  using Theorem~\ref{th-approx}, the case of a general continuous
  function with compact support will follow easily.  We write
$$
\Er[f(X_n)]=
\int_{\Rr^d}{f(x)d\mu_n(x)}
$$
where $\mu_n$ is the law of $X_n$, before applying the Plancherel
formula and the inversion formula to transform this into
$$
\Er[f(X_n)]=
\frac{1}{(2\pi)^d}
\int_{\Rr^d} \varphi_n(t)\hat{f}(t)\,m(dt).
$$
\par
By the linear change of variable $t=\Sigma_n^*s$, we get
$$
\Er[f(X_n)]=(2\pi)^{-d}|\det(\Sigma_n)|\int_{\Rr^d}
\varphi_n(\Sigma_n^*s)\hat{f}(\Sigma_n^*s)\,m(ds).
$$
\par
Now fix $k$ so that the support of $\hat{f}$ is contained in the ball
of radius $k$; we then have
$$
\Er[f(X_n)]=(2\pi)^{-d}|\det(\Sigma_n)|\int_{|\Sigma_n^*s|\leq k}
\varphi_n(\Sigma_n^*s)\hat{f}(\Sigma_n^*s)\,m(ds).
$$
\par
The integrand converges pointwise to $\varphi(s)\hat{f}(0)$ according
to the assumption {\bf H2}. The condition {\bf H3} of uniform
integrability then implies the convergence in $L^1$. One can see this
quickly in this case: for any $\ep>0$, and for any $a>0$ large enough,
we have
$$
\int_{|s|>a}{
|\varphi_n(\Sigma_n^*s)\one_{|\Sigma_n^*s|\leq k}
\hat{f}(\Sigma_n^*s)|\,m(ds)}
\leq \|\hat{f}\|_{\infty}\int_{|s|>a}{
|\varphi_n(\Sigma_n^*s)\one_{|\Sigma_n^*s|\leq k}|\,m(ds)}
<\ep
$$
for all $n$ by~(\ref{eq-formal-h3}). On $|s|\leq a$, the pointwise
convergence is dominated by $\|\hat{f}\|_{\infty}\one_{|s|\leq a}$,
hence
$$
\int_{|s|\leq a}{ \varphi_n(\Sigma_n^*s)\one_{|\Sigma_n^*s|\leq k}
  \hat{f}(\Sigma_n^*s)\,m(ds)}\rightarrow \hat{f}(0)\int_{|s|\leq
  a}{\varphi(s)m(ds)}.
$$
\par
For $a$ large enough, this is $\hat{f}(0)\int{\varphi}$, up to error
$\ep$, hence we get the convergence
$$
\int_{|\Sigma_n^*s|\leq k}
\varphi_n(\Sigma_n^*s)\hat{f}(\Sigma_n^*s)\,m(ds) \rightarrow
\hat{f}(0)\int_{\Rr^d}{\varphi(s)m(ds)}.
$$
\par
Finally, this leads to
$$
|\det(A_n)|\Er[f(X_n)]\rightarrow (2\pi)^{-d}\hat{f}(0)\int_{\Rr^d}
\varphi(s)m(ds)=
\frac{d\mu}{dm}(0)\int_{\Rr^d}{f(s)m(ds)},
$$
which concludes the proof for $f$ with $\hat{f}$ with compact
support.
\par
Now if $f$ is continuous with compact support, we use
Theorem~\ref{th-approx}: by linearity, we can assume $f$ to be
real-valued, and then, given $\eta>0$ and $g_2\leq f\leq g_1$ as in
the approximation theorem, we have
$$
|\det(A_n)|\Er[g_2(X_n)]\leq |\det(A_n)|\Er[f(X_n)]\leq 
|\det(A_n)|\Er[g_1(X_n)],
$$
and hence
$$
0\leq \limsup_{n} |\det(A_n)|\Er[f(X_n)]-
\liminf_{n} |\det(A_n)|\Er[f(X_n)]\leq 
\int{(g_1-g_2)dm}\leq \eta,
$$
which proves the result since $\eta>0$ is arbitrary. Similarly, the
proof of~(\ref{eq-limit-jordan}) is performed in standard ways.
\end{proof}

\begin{rem} To illustrate why our results are generalisations of the
  local theorems, let us analyse a particularly simple situation. We
  assume that $d=1$ and that the random variables $X_n$ have
  characteristic functions $\varphi_n$ such that $\varphi_n(t/b_n)$
  converge to $\varphi(t)$ in $L^1(\Rr)$, with $b_n\rightarrow
  +\infty$ (such situations are related, but less general, than the
  classical results discussed in Section~\ref{sec-central-limit} or in
  \cite{BDC} and \cite{St}.) In that case, the density functions $f_n$
  of $X_n/b_n$ exist, are continuous and converge (in $L^1(\Rr)$ and
  uniformly) to a continuous density function $f$. For a bounded
  interval $(\alpha,\beta)$, we obtain
\begin{align*}
  b_n\Pr[X_n\in(\alpha,\beta)]&=b_n
  \Pr[X_n/b_n\in (\alpha/b_n,\beta/b_n)]\\
  &=b_n\int_{\alpha/b_n}^{\beta/b_n} f_n(x)\,dx\rightarrow
  f(0)(\beta-\alpha),
\end{align*} 
by elementary calculus. 
\end{rem}

It may be worth remarking explicitly that it is quite possible for
this theorem to apply in a situation where the constant
$\frac{d\mu}{dm}(0)$ is zero. In this case, the limit gives some
information, but is not as precise as when the constant is non-zero.
For instance, consider the characteristic function
$\varphi(t)=1/(1-it)^2$, which corresponds to the sum $E_1+E_2$ of two
independent exponential random variables with density $e^{-x}dx$ on
$[0,+\infty[$. An easy computation shows that the density for
$\varphi$ itself is $xe^{-x}$ (supported on $[0,+\infty[$), and for
$X_n=n(E_1+E_2)$, we have mod-$\varphi$ convergence with $A_nt=nt$,
leading to the limit
$$
\lim_{n\rightarrow +\infty} n\Pr[\alpha<X_n<\beta]=0
$$
for all $\alpha<\beta$. Note that any other limit $c(\beta-\alpha)$
would not make sense here, since $X_n$ is always non-negative, whereas
there is no constraint on the signs of $\alpha$ and $\beta$...
\par
However, in similar cases, the following general fact will usually
lead to more natural results:

\begin{prop}[Mod-$\varphi$ convergence and shift of the mean]\label{pr-shift-mean}
  Let $d\geq 1$ be an integer, and let $(X_n)$ be a sequence of
  $\Rr^d$-valued random variables such that there is mod$-\varphi$
  convergence with respect to the linear maps $A_n$. Let
  $\alpha\in\Rr^d$ be arbitrary, and let $\alpha_n\in\Rr^d$ be a
  sequence of vectors such that
\begin{equation}\label{eq-calibrate-shift}
  \lim_{n\rightarrow +\infty} \Sigma_n\alpha_n=\alpha,
\end{equation}
for instance $\alpha_n=A_n\alpha$. Then the sequence
$Y_n=X_n-\alpha_n$ satisfies mod-$\psi$ convergence with parameters
$A_n$ for the characteristic function
$$
\psi(t)=\varphi(t)e^{-it\cdot \alpha}.
$$
\par
In particular, for any continuous function $f$ on $\Rr^d$ with compact
support, we have
$$
\lim_{n\rightarrow+\infty} |\det(A_n)|\Er[f(X_n-\alpha_n)]=
\frac{d\mu}{dm}(\alpha)\int_{\Rr^d}{f(x)m(dx)},
$$
where $\mu$ is the probability measure with characteristic function
$\varphi$, and for any bounded Jordan-measurable subset $B\subset
\Rr^d$, we have
\begin{equation}\label{eq-shifted-lim}
\lim_{n\rightarrow+\infty} |\det(A_n)|\Pr[X_n-\alpha_n\in B]=
\frac{d\mu}{dm}(\alpha)
m(B).
\end{equation}
\end{prop}

\begin{proof}
  This is entirely elementary: $\psi$ is of course integrable and
  since
$$
\Er[e^{itY_n}]=\varphi_n(t)e^{-it\cdot \alpha_n},
$$
we have $\Er[e^{i\Sigma_n^*t\cdot Y_n}]
=\varphi_n(\Sigma_n^*t)e^{-it\cdot \Sigma_n\alpha_n}$, which converges
locally uniformly to $\psi(t)$ by our
assumption~(\ref{eq-calibrate-shift}). Since the modulus of the
characteristic function of $Y_n$ is the same, at any point, as that of
$X_n$, Property {\bf H3} holds for $(Y_n)$ exactly when it does for
$(X_n)$, and hence mod-$\psi$ convergence holds. If $f=d\mu/dm$, the
density of the measure with characteristic function $\psi$ is
$g(x)=f(x+\alpha)$, and therefore the last two limits hold by
Theorem~\ref{th-local-limit}.
\end{proof}

In the situation described before the statement, taking $\alpha_n=cn$
with $c>0$ leads to the (elementary) statement
$$
\lim_{n\rightarrow +\infty} n\Pr[\alpha+cn<X_n<
\beta+cn]=ce^{-c}(\beta-\alpha).
$$
\par
Even when the density of $\mu$ does not vanish at $0$, limits
like~(\ref{eq-shifted-lim}) are of interest for all $\alpha\not=0$.
\par
Another easy and natural extension of the local limit theorem involves
situations where a further linear change of variable is performed:

\begin{prop}[Local limit theorem after linear change of variable]
  Suppose that $(X_n)$ satisfies mod-$\varphi$ convergence relative to
  $A_n$ and $\Sigma_n$.  Suppose that $(T_n)$ is a sequence of linear
  isomorphisms $T_n\colon \Rr^d\rightarrow\Rr^d$ such that
  $T_n^{-1}\rightarrow 0$ and $\Sigma_nT_n\rightarrow 0$. Suppose also
  that the following balancedness condition holds: there is a constant
  $C$ such that $|(\Sigma_nT_n)^*t|\le 1$ implies that $|\Sigma^*_n
  t|\le C$. Then the sequence $T^{-1}_nX_n$ also satisfies the
  conditions of the theorem, and in particular for any bounded Jordan
  measurable set $B$ we have
$$
\frac{|\det(A_n)|}{|\det(T_n)|}\,\Pr\[X_n\in T_nB\]\rightarrow
\frac{d\mu}{dm}(0)\,m(B).
$$
\end{prop}

\begin{proof} Let us put $\tilde{X}_n=T_n^{-1}X_n$ and
$\tilde{\varphi}(t)=\Er[\exp(i\,t.T_n^{-1}X_n)]=\varphi_n((T_n^{-1})^*t)$.
Clearly the sequence $\Sigma_nT_n$ tends to zero and
$\Sigma_nT_n(\tilde{X}_n)=\Sigma_nX_n$ tends to $\mu$ in law.  The
only remaining thing to verify is the uniform integrability
condition. Let us look at
$$
\tilde{\varphi}_n((\Sigma_nT_n)^*t)\one_{|(\Sigma_nT_n)^*t|\le
  k}=\varphi_n((\Sigma_n)^*t)\one_{|(\Sigma_nT_n)^*t|\le k}.$$ Because
of the balancedness condition we get that
$$\one_{|(\Sigma_nT_n)^*t|\le k}\le \one_{|(\Sigma_n)^*t|\le C k}.$$
\par
The rest is obvious.
\end{proof}

\begin{rem} The balancedness condition is always satisfied if
  $d=1$.  In dimension $d>1$, there are counterexamples.  In case the
  ratio of the largest singular value of $\Sigma_n$ to its smallest
  singular value is bounded, the balancedness condition is satisfied
  (this is an easy linear algebra exercise). See also \cite{KN} for
  the use of such conditions in mod-gaussian convergence.  To see that
  for $d=2$ it is not necessarily satisfied take the following
  sequences:
$$
\Sigma_n=\(\begin{array}{cc}n^{-1/4} & 0\\  0 &n^{-1/2}
\end{array}\),\quad\quad T_n=\(\begin{array}{cc}0 &n^{1/8}\\ 
  n^{1/8} &0\end{array}\).
$$
\end{rem}

\subsection{Conditions ensuring mod-$\varphi$ convergence}

We now derive other equivalent conditions, or sufficient ones, for
mod-$\varphi$ convergence. First of all, the conditions {\bf H1}, {\bf
  H2}, {\bf H3} have a probabilistic interpretation. We suppose $d=1$
to keep the presentation simple. Instead of taking the indicator
function $\one_{|\Sigma_n^*t|\le k}$, we could have taken the
triangular function $\Delta_k$ defined as
$\Delta_k(0)=1,\Delta_k(2k)=0=\Delta(-2k)$, $\Delta_k(x)=0$ for
$|x|\ge 2k$ and $\Delta_k$ is piecewise linear between the said
points. The function $\Delta_1$ is the characteristic function of a
random variable $Y$ (taken independent of the sequence $X_n$). Hence
we get that the sequence $X_n$ satisfies ${\bf H1,H2,H3}$ if and only
if for each $k\ge 1$, the characteristic functions of
$Z_n=\Sigma_n\(X_n+\frac{1}{k}Y\)$ converge in $L^1(\Rr^d)$ to
$\varphi$. Indeed the characteristic function of $Z_n$ equals
$\varphi_n(\Sigma_n t)\Delta_k(\Sigma_n t)$.  There is no need to use
the special form of the random variable $Y$.
\par
In fact we have the following:

\begin{theo} Suppose that for the sequence $X_n$ the conditions {\bf
    H1,H2} hold.  The condition {\bf H3} holds as soon as there is a
  random variable, $V$, independent of the sequence $X_n$ such that
  for each $\ep>0$, $\Er\[\exp\(i t \Sigma_n\(X_n+\ep V\)\)\]$ tends
  to $\varphi$ in $L^1(\Rr^d)$.
\end{theo}

\begin{proof} 
  Let $\psi(t)=\Er\[\exp\(i t\cdot V\)\]$ be the characteristic
  function of $V$. The hypothesis of the theorem is equivalent to the
  property that for each $\ep>0$, the sequence
$$
\varphi_n(\Sigma_n^* t)\psi(\ep \Sigma_n^* t),
$$
is uniformly integrable. Let $\delta>0$ be such that for $|t|\le
\delta$, $|\psi(t)|\ge 1/2$.  Then the uniform integrability of the
above mentioned sequence implies for each $\ep>0,$ the uniform
integrability of the sequence
$$
|\varphi_n(\Sigma_n^* t)| \one_{\ep|\Sigma_n^*t|\le\delta}\le
2\left|\varphi_n(\Sigma_n^* t)\psi(\ep \Sigma_n^* t)\right|.
$$
\par
This ends the proof.
\end{proof}

\begin{rem}\rm 
  We suppose that $d=1$.  For higher dimensions the discussion can be
  made along the same lines but it is much more tricky.  Polya's
  theorem says that if $\gamma\colon \Rr_+\rightarrow \Rr_+$ is a
  convex function such that $\gamma(0)=1$, $\lim_{x\rightarrow
    +\infty}\gamma(x)=0$, then there is a random variable $Y$ (which
  can be taken to be independent of the sequence $X_n$, such that
  $\varphi_Y(t)=\gamma(|t|)$.  The characteristic function of
  $\Sigma_n(X_n+Y)$ is then $\gamma(|\Sigma_n t|)\varphi_n(\Sigma_n
  t)$ and hence is a uniformly integrable sequence. Since
  $\Vert\Sigma_n\Vert\rightarrow 0$, we see that $\Sigma_n(X_n+Y)$
  tends in law to $\mu$ with characteristic function $\varphi$. The
  convergence is much stronger than just weak convergence.  In fact
  the random variables $\Sigma_n(X_n+Y)$ have densities and because
  the characteristic functions tend in $L^1$ to $\varphi$, the
  densities of $\Sigma_n(X_n+Y)$ converge to the density of $\mu$ in
  the topology of $L^1(\Rr)$.
\end{rem}

\begin{rem}\rm Adding a random variable $Y$ can be seen as a
  regularisation (mollifier). Indeed adding an independent random
  variable leads to a convolution for the densities.  In our context
  this means that the distribution of $X_n$ is convoluted with an
  integrable kernel (the density of $Y$).  The regularity of the law
  of $Y$ is then passed to the law of $X_n+Y$.  In probability theory
  such a mollifier is nothing else than adding an independent random
  variable with suitable properties.
\end{rem}

We can go one step further and replace the condition for each $k$ by a
condition where we use just one random variable.  This is the topic of
the next theorem.

\begin{theo}\label{th-modphi-equivalent} 
  Suppose that the hypotheses {\bf H1,H2} hold. Then {\bf H3} is also
  equivalent to either of the following:
\par
\emph{(1)} There exists a non-increasing function $\gamma\colon
\Rr_+\rightarrow \Rr_+$, such that $\gamma(0)=1$, $0<\gamma\le 1$,
$$
\lim_{x\rightarrow +\infty}\gamma(x)=0
$$
and such that the sequence $\gamma(|\Sigma_n^*t|)\varphi(\Sigma_n^*t)$
is uniformly integrable.
\par
\emph{(2)} There exists a non-increasing convex function $\gamma\colon
\Rr_+\rightarrow \Rr_+$, such that $\gamma(0)=1$, $0<\gamma\le 1$, 
$$
\lim_{x\rightarrow +\infty}\gamma(x)=0
$$
and such that the sequence $\gamma(|\Sigma_n^*t|)\varphi(\Sigma_n^*t)$
is uniformly integrable.
\end{theo}

\begin{proof} It is quite clear that (1) or (2) imply {\bf H3}, since
for any $k>0$, we obtain
$$
|\varphi_n(\Sigma_n^* t)|\one_{|\Sigma_n^* t|\le k}\le
\frac{1}{\gamma(k)}\gamma(|\Sigma_n^*t|)|\varphi(\Sigma_n^*t)|,
$$
and therefore the desired uniform integrability.
\par
For the reverse, it is enough to show that {\bf H3} implies (2), since
(1) is obviously weaker. For $x\ge 0$ we define
$$
g(x)=\sup_n\Vert\varphi_n(\Sigma_n^*t)\one_{|\Sigma_n^* t|\le
  x+1}\Vert_{L^1(\Rr^d)}.
$$ 
\par
The function $g$ is clearly non-decreasing.  Let us first observe that
there is a constant $C>0$ such that for $x$ big enough,
$\exp\(-\int_0^xg(s)\,ds\)\le \exp(-Cx)$. We define
$$
\gamma(x)=\alpha \int_x^\infty\exp\(-\int_0^u g(s)\,ds\)\,du,
$$ 
where $\alpha$ is chosen so that $\gamma(0)=1$ (since the integrals
converge, this function is well defined).
\par
The function $\gamma$ is also convex and tends to zero at $\infty$.
Furthermore $$\gamma(x)\le \alpha \int_x^\infty
\frac{g(u)}{g(x)}\exp\(-\int_0^u g(s)\,ds\)\,du \le
\frac{\alpha}{g(x)}\exp\(-\int_0^xg(s)\,ds\)$$ from which it follows
that $\int_0^\infty \gamma(x) g(x)\,dx<\infty$. We now claim that
$\gamma(|\Sigma_n^*t|)\varphi_n(\Sigma_n^* t)$ is uniformly
integrable. Because the sequence is uniformly bounded we only need to
show that for each $\ep>0$ there is a $k$ such that
$$
\int_{\Rr^d} \gamma(|\Sigma_n^*t|)|\varphi_n(\Sigma_n^*
t)|\one_{|t|\ge k}\,m(dt)\le \ep
$$
for all $n\geq 1$.
\par
For $k,K$ integers, we split the integral as follows
\begin{align*}
  \int_{\Rr^d} \gamma(|\Sigma_n^*t|)|\varphi_n(\Sigma_n^*
  t)|\one_{|t|\ge k}\,m(dt)
  &\le \int_{\Rr^d} \gamma(|\Sigma_n^*t|)|\varphi_n(\Sigma_n^* t)|\one_{|t|\ge k}\one_{|\Sigma_n^*t|\le K}\,m(dt)\\
  &\quad\quad + \int_{\Rr^d}\gamma(|\Sigma_n^*t|)|\varphi_n(\Sigma_n^* t)|\one_{|t|\ge k}\one_{|\Sigma_n^*t| > K}\,m(dt)\\
  &\le \int_{\Rr^d} |\varphi_n(\Sigma_n^* t)|\one_{|t|\ge k}\one_{|\Sigma_n^*t|\le K}\,m(dt)\\
  &\quad\quad + \int_{\Rr^d}\gamma(|\Sigma_n^*t|)|\varphi_n(\Sigma_n^*
  t)|\one_{|\Sigma_n^*t| > K}\,m(dt).
\end{align*}
\par
The last term is dominated as follows:
\begin{align*}
  \int_{\Rr^d}\gamma(|\Sigma_n^*t|)|\varphi_n(\Sigma_n^* t)|\one_{|\Sigma_n^*t| > K}\,m(dt)
  &\le \sum_{l\ge K} \gamma(l)\int_{l\le |\Sigma_n^* t|\le l+1}|\varphi_n(\Sigma_n^* t)|\,m(dt)\\
  &\le \sum_{l\ge K}\gamma(l) g(l)\\
  &\le \alpha\sum_{l\ge K}\exp\(-\int_0^l g(s)\,ds\),
\end{align*}
which can be made smaller than $\ep/2$ by taking $K$ big enough.  Once
$K$ fixed we use the uniform integrability of the sequence
$\varphi_n(\Sigma_n^* t)|\one_{|\Sigma_n^*t|\le K}$ and take $k$ big
enough so that we get for each $n$:
$$
\int_{\Rr^d} |\varphi_n(\Sigma_n^* t)|\one_{|t|\ge k}
\one_{|\Sigma_n^*t|\le K}\,m(dt)\le \ep/2.
$$
This completes the proof.  \end{proof}

In particular, we get a sufficient condition:

\begin{cor}\label{cor-sufficient-condition} 
  Suppose that the sequence $X_n$ satisfies the following:
\begin{enumerate}
\item {\bf H1,H2} hold;
\item There is a non-decreasing function $c\colon
  \Rr\rightarrow\Rr_+$, $c(0)=1$ as well as an integrable function
  $h\colon \Rr^d\rightarrow\Rr_+$ such that $|\varphi_n(t)|\le
  h(A_n^*t)\,c(|t|)$ for all $n$ and $t\in\Rr^d$.
\end{enumerate}
\par
Then the property {\bf H3} holds as well.
\end{cor}

\begin{proof} This is clear from the previous theorem, since
  $|\varphi_n(t)|\le h(A_n^*t)\,c(|t|)$ for all $t$ implies that for
  all $t$ we have
$$
\frac{|\varphi_n(\Sigma_n^*t)|}{c(|\Sigma_n^*(t)|)}\le h(t),
$$
which verifies (1) in Theorem~\ref{th-modphi-equivalent}.
\end{proof}


\section{Applications}

In this section, we collect some examples of mod-$\varphi$
convergence, for various types of limits $\varphi$, and therefore
derive local limit theorems. Some of these results are already known,
and some are new. It is quite interesting to see all of them handled
using the relatively elementary framework of the previous section. The
coming subsections are mostly independent of each other; the first few
are of probabilistic nature, while the last ones involve arithmetic
considerations. 

\subsection{The Central Limit Theorem and convergence to stable laws}\label{sec-central-limit}

In this section we suppose that $(X_n)_{n\ge 1}$ is a sequence of
independent identically distributed random variables.  The central
limit theorem deals with convergence in law of expressions of the form
$\frac{X_1+\ldots+X_n}{b_n}-a_n$, where $b_n$ are normalising
constants. We will suppose without further notice that the random
variables are {\it symmetric} so that we can suppose $a_n=0$.  The
possible limit laws have characteristic functions of the form
$\exp\(-c\,|u|^p\)$, where $0< p\le 2$ and where $c>0$.  For
information regarding this convergence we refer to
Lo\`eve~\cite{Lo}. The basis for the theory is Karamata's theory of
regular variation. In this section we are interested in expressions of
the form $\lim b_n\Pr[(X_1+\ldots+X_n)\in B]$ for suitably bounded
Borel sets $B$.
\par
For the case $\Er[X^2]<\infty$, the problem was solved by
Shepp~\cite{Sh}.  The multidimensional square integrable case was
solved by Borovkov and Mogulskii~\cite{BM} and Mogulskii~\cite{Mo}.
The case $p<2$ was solved by Stone~\cite{St} and at the same by
Bretagnolle and Dacunha-Castelle~\cite{BDC} (see also Ibragimov and
Linnik~\cite{IL}). Such theorems are known as \emph{local limit
  theorems}.

\begin{theo}\label{th-classical} 
  Suppose that the non-lattice random variable $X$ is symmetric and
  that it is in the domain of attraction of a stable law with exponent
  $p$. More precisely we suppose that $\frac{X_1+\ldots X_n}{b_n}$
  converges in law to a probability distribution with characteristic
  function $\exp(-|t|^p)$, $0<p\le 2$. Then for Jordan-measurable
  bounded Borel sets, we have
$$
\lim_{n\rightarrow+\infty} 
b_n\Pr[(X_1+\ldots+X_n)\in B]=c_p\, m(B),
$$ 
where $c_p=\frac{1}{2\pi}\int_{-\infty}^{+\infty}\exp(-|t|^p)\,dt$.
Suppose moreover that $0<\tau_n\rightarrow +\infty$ in such a way that
$\frac{b_n}{\tau_n}\rightarrow +\infty$, then
$$
\lim_{n\rightarrow +\infty} 
\frac{b_n}{\tau_n}\Pr\[\frac{X_1+\ldots+X_n}{\tau_n}\in B\]
= c_p\, m(B).
$$
\end{theo} 

In order to prove this theorem, we first observe that, when {\bf H1}
and {\bf H2} are satisfied, the condition {\bf H3} of uniform
integrability is equivalent with classical conditions that arise in
the current context.

\begin{theo}
  Under the hypotheses {\bf H1,H2}, the hypothesis {\bf H3} is
  equivalent to the validity of the following two conditions:
\begin{itemize}
\item{\bf H3'.} For all $k\ge \ep>0$, we have
$$
\lim_{n\rightarrow +\infty}|\det(A_n)|\,\int_{\ep\le
    |t|\le k} |\varphi_n(t)|\,m(dt)=0.
$$
\item{\bf H4'.} For all $\eta>0$. there is $a\ge 0, \ep>0$ such that
$$\limsup_{n\rightarrow +\infty} \int_{a\le|t|;|\Sigma^*_nt|\le
  \ep}|\varphi_n(\Sigma_n^*t)|\,m(dt)\le\eta.
$$
\end{itemize}
\end{theo}

\begin{proof} First suppose that {\bf H3} holds, i.e., for each $k>0$,
$\varphi_n(\Sigma_n^*t)\one_{|\Sigma_n^*t|\le k}$ is uniformly
integrable. Since $\Sigma_n\rightarrow 0$, we immediately
get 
$$
\lim_n\int_\ep^k
|\det(A_n)\varphi_n(t)|\,m(dt)=\int_{\ep\le|\Sigma_n^*s|\le
  k}|\varphi_n(\Sigma_n^*s)|\,m(ds)\rightarrow 0,
$$ 
which is {\bf H3'}. To establish {\bf H4'}, let us first remark that
(using {\bf H2}) we have
$$
\lim_n \varphi_n(\Sigma_n^*t)\one_{|\Sigma_n^*t|\le k}=\varphi(t)
$$
for all $t$. Then we take $a>0$ such that for given $\eta>0$ we
have $$\int_{|t|\ge a}|\varphi(t)|\,m(dt)\le \eta.$$ Take now $\ep>0$
and observe that by uniform integrability 
$$
\lim_n \int_{a\le|t|;|\Sigma^*_nt|\le
  \ep}|\varphi_n(\Sigma_n^*t)|\,m(dt)=\int_{|s|\ge
  a}|\varphi(s)|\,m(ds)\le\eta.
$$ 
\par
Now we proceed to the converse and we suppose that {\bf H1,H2,H3',H4'}
hold. We first show that given $\eta>0$, the sequence
$\varphi_n(\Sigma_n^*t)\one_{|\Sigma_n^*t|\le k}$ has up to $\eta$ all
its mass on a ball of radius $a$. Given $\eta>0$ we can find $a,\ep>0$
such that 
$$
\lim_n \int_{a\le|t|;|\Sigma^*_nt|\le
  \ep}|\varphi_n(\Sigma_n^*t)|\,m(dt)\le\eta.
$$ 
\par
Then according to {\bf H3'} and {\bf H4'}, we can find $n_0$ such that
for all $n\ge n_0$ we have
\begin{gather*}
\int_{a\le|t|;|\Sigma^*_nt|\le \ep}|\varphi_n(\Sigma_n^*t)|\,m(dt)
\le 2\eta,\\
\int_{\ep\le|\Sigma_n^*t|\le k}|\varphi(\Sigma_n^*t)|\,m(dt)\le \eta.
\end{gather*}
\par
Increasing $a$ allows us to suppose that the same inequalities hold
for all $n\ge 1$. So we get that 
$$
\int_{|\Sigma_n^*t|\le k; |t|\ge a}|\varphi(\Sigma_n^*t)|\,m(dt)\le
3\eta.
$$ 
\par
Since the sequence is uniformly bounded we have proved uniform
integrability.
\end{proof}


\begin{proof}[Proof of Theorem~\ref{th-classical}] We have here
  $\varphi_n=\psi^n$ where $\psi$ is the characteristic function of a
  random variable in the domain of attraction of a stable law.
  Property {\bf H3'} follows since the sequence $\varphi_n$ tends to
  zero exponentially fast, uniformly on compact sets of
  $\Rr^d\setminus\{0\}$.  Moreover, Property {\bf H4'} is known as
  Gnedenko's condition (see Gnedenko and Kolmogorov~\cite{GK} or the
  discussion of $I_2$ (resp. $I_3$) in Ibragimov and
  Linnik~\cite[p. 123]{IL}) (resp.~\cite[p. 127]{IL}). Thus the
  hypotheses {\bf H1}, {\bf H2}, {\bf H3'}, {\bf H4'} are fulfilled in
  this setting.
\end{proof}

\begin{rem}\rm 
(1)  The proof of Property {\bf H4'} is based on the regular variation of
  $\psi$ around $0$.  The fact that regular variation is needed
  suggest that it is difficult to get a more abstract version of this
  property.
\par
(2) Taking the most classical case where $p=2$ and $(X_n)$ independent
and identically distributed, it is easy to check that the stronger
condition {\bf H2'} (i.e.,~(\ref{eq-original}) is \emph{not}
valid, except if the $X_n$ are themselves gaussian random
variables. Thus the setting in this paper is a genuine generalization
of the original mod-gaussian convergence discussed in~\cite{JKN}.
\end{rem}

\subsection{The winding number of complex Brownian motion}
\label{sec-winding}

We take a complex Brownian Motion $W$, starting at $1$.  Of course we
can also see $W$ as a two-dimensional real BM.  The process $W$ will
never attain the value $0$ and hence, by continuous extension or
lifting, we can define the argument $\theta$. We get
$W_u=R_u\exp(i\theta_u)$ where $\theta_0=0$ and $R_u=|W_u|$.  The
process $\theta$ is called the \emph{winding number}, see~\cite{MY}.
Spitzer in \cite{Sp} computed the law of $\theta_u$ and gave its
Fourier transform, and a more precise convergence result was given in
\cite{BPY}.  
\par
The characteristic function is given by
\begin{align*}
  \Er\[\exp(it\theta_u)\]= \(\frac{\pi}{2}\)^{1/2}
  \(\frac{1}{4u}\)^{1/2}\exp\(-\frac{1}{4u}\)
\(I_{(|t|-1)/2}\(\frac{1}{4u}\)+I_{(|t|+1)/2}\(\frac{1}{4u}\)\)
\end{align*}
where $I_{\nu}(z)$ denotes the $I$-Bessel function, which can be
defined by its Taylor expansion
$$
I_{\nu}(z)=\sum_{m\geq 0}{\frac{1}{m!\Gamma(\nu+m+1)}
\Bigl(\frac{z}{2}\Bigr)^{\nu+2m}}.
$$
\par
Using elementary properties of Bessel functions, Spitzer deduced that
$\frac{2\theta_u}{\log u}$ converges to a Cauchy law with
characteristic function $\varphi(t)=\exp(-|t|)$ and density
$\frac{1}{\pi}\frac{1}{1+t^2}$.

\begin{theo}[Mod-Cauchy convergence of the winding number]
  For any sequence $(u_n)$ of positive real numbers tending to
  infinity, the sequence $X_n=\theta_{u_n}$ satisfies mod-$\varphi$
  convergence with $d=1$, $\varphi(t)=\exp(-|t|)$,
  $A_n(t)=A_n^*(t)=(\log u_n)t/2$.
\par
In particular, for any real numbers $a<b$, we have
$$
\lim_{u\rightarrow \infty}\frac{\log u}{2}\,
\Pr\[a<\theta_u<b\]=\frac{1}{\pi}(b-a).
$$
\end{theo}

Although this is a very natural statement, we have not found this
local limit theorem in the literature.

\begin{proof} The conditions {\bf H1} and {\bf H2} of mod-$\varphi$
  convergence are clear, the second by Spitzer's Theorem. To check the
  uniform integrability condition {\bf H3}, we take $k\geq 1$ and we
  proceed to bound
$$
|\varphi_n(\Sigma_n^*t)|\one_{|\Sigma_n^*t|\leq k}
$$
for $t\geq 0$. But if $|\Sigma_n^*t|\leq k$, we have
$$
\frac{1}{2}\leq \frac{|\Sigma_n^*t|\pm 1}{2}\leq \frac{k+1}{2}
$$
and $0\leq \frac{1}{4u_n}\leq 1$ for $n$ large enough. The Taylor
series expansion shows immediately that there exists $C_k\geq 0$ such
that
$$
|I_{\nu}(z)|\leq C_k
$$
uniformly for $\nu$ real with $-1/2\leq \nu\leq \frac{k+1}{2}$ and
$z\in\Cr$ with $|z|\leq 1$, so that for $|\Sigma_n^*t|=2|t|/(\log
u_n)\leq k$, we have
$$
|\varphi_n(\Sigma_n^*t)|\leq B_ku_n^{-1/2}\leq
B_k\exp\Bigl(-\frac{|t|}{k}\Bigr)
$$
where $B_k=C_k(\pi/2)^{1/2}$. This gives the desired uniform
integrability, in the form~(\ref{eq-dominated}).
\end{proof}

\subsection{``Relaxed'' Poisson variables}

We present here a special case of a phenomenon which is related to
Poisson approximation and therefore probably quite general: if $P_n$,
$n\geq 1$, denotes a Poisson-distributed random variable with
parameter $\lambda_n$ going to infinity, the sequence
$$
X_n=\frac{P_n-\lambda_n}{\lambda_n^{1/3}}
$$
satisfies mod$-\varphi$ convergence with $d=1$,
$\varphi(t)=e^{-t^2/2}$ (i.e., for a standard gaussian) and
$A_nt=\lambda_n^{1/6}t$. Indeed, {\bf H2} holds because
$$
\varphi_n(t)=e^{-i\lambda_n^{2/3}t}\exp(\lambda_n
(e^{it/\lambda_n^{1/3}}-1))
=\exp(-\lambda_n^{1/3}t^2/2)\exp(-it^3/6)(1+o(1))
$$
as $n$ tends to infinity. Moreover, the next term in the expansion of
the exponential $e^{it/\lambda_n^{1/3}}$ shows that 
$$
\Bigl|\varphi_n\(\frac{t}{n^{1/6}}\)\Bigr|=e^{-t^2/2}
\(1+O\(\frac{|t|^4}{\lambda_n}\)\)
$$
and the uniform integrability condition {\bf H3} therefore holds even
for the range $|t|\leq \lambda_n^{1/4}$. (Except for {\bf H3}, this
example was considered in~\cite[Prop. 2.4]{KN2}.)
\par
As a consequence, we get the local limit
$$
\lim_{n\rightarrow +\infty}
\lambda_n^{1/6}\Pr[a\lambda_n^{1/3}+\lambda_n<
P_n<b\lambda_n^{1/3}+\lambda_n]=\frac{b-a}{\sqrt{2\pi}}
$$
for any fixed $a<b$.

\begin{rem}
  Using the formula~\cite[(4.9)]{KN2}, we see that the same
  mod-$\varphi$ convergence property holds when $P_n$ is replaced with
  $\omega_n$ defined as the number of cycles in the decomposition in
  cycles of a uniformly chosen random permutation in the symmetric
  group on $n$ letter, with $\lambda_n=\log n$. These are well-known
  (see, e.g.,~\cite{abt, KN2}) to be well-approximated by Poisson
  variables with these parameters. 
\end{rem}

\subsection{Dedekind Sums}

In this section we give an application to Dedekind sums.  Our limit
theorems are based on the estimates in Vardi's paper \cite{Va}. Let us
recall the definition of Dedekind sums. We recall the standard
notation
\begin{align*}
\lfloor x\rfloor&=\sup\{n\in\Zr\mid n\le x\},\\
\ddl x\ddr&=x-\lfloor x\rfloor -1/2,\text{ if }x\notin \Zr\\
&=0\text{ if }x\in\Zr.
\end{align*}
\par
For natural numbers $0<d<c$ with $\gcd(d,c)=1$, the Dedekind sum is
defined as
$$
s(d,c)=\sum_{0<k<c}\dl\frac{kd}{c}\dr\dl\frac{d}{c}\dr.
$$
For every $N\in \Nr$ we define the finite probability space:
\begin{align*}
  \Omega_N&=\{(d,c)\mid 0<d<c<N;\gcd(d,c)=1\},\\
  \Pr_N[A]&=\frac{\#A}{\#\Omega_N} \text{
    the normalised counting measure},\\
  X_N(d,c)&=s(d,c).
\end{align*}
\par
The distribution of $X_N$ is symmetric as easily seen by using the
measure preserving transformation $(d,c)\rightarrow (c-d,c)$. It is
well known that $\#\Omega_N/\(\frac{3N^2}{\pi^2}\)\rightarrow 1$, see
e.g section 3.4 of this paper. Vardi~\cite[Prop. 2]{Va} proved an
asymptotic formula which implies the following:

\begin{prop} For $0\le |t|\le 1/4$ we have that:
$$
|\varphi_N(2\pi t)|\le CN^{-|t|}+o(N^{-1/3}),
$$
where $C$ is an absolute constant and where the last term is uniform
in $t$.
\end{prop}

\begin{rem}\rm The result of~\cite{Va} actually gives the same
  result for larger values of $t$, but the error term is only smaller
  than the main term when $|t|<2/3$.
\end{rem}

As a consequence of the same proposition in \cite{Va}, we get that for
$t\in\Rr$:
$$
\varphi_N\(\frac{2\pi t}{\log N}\)\rightarrow \varphi(t)=\exp(-|t|),
$$
the characteristic function of a standard Cauchy random variable with
density $\frac{1}{\pi(1+x^2)}$.
\par
The bound given by Vardi does not allow to show a mod$-\varphi$ (in
this case ``mod-Cauchy'') convergence, but it suffices to obtain the
following weaker statement:

\begin{prop} 
  For any sequence $(\tau_N)$ such that $\tau_N\rightarrow +\infty$
  and $\frac{\log N}{\tau_N}\rightarrow+\infty$, the sequence
  $\frac{X_N}{\tau_N}$ satisfies mod-$\varphi$ convergence with
  $A_Nt=\frac{\log N}{2\pi \tau_N}t$.
\par
Hence, for every bounded Jordan-measurable set $B\subset \Rr$, we have
$$
\frac{\log N}{2\pi \tau_N}\Pr_N\[\frac{X_N}{\tau_N}\in
B\]\rightarrow \frac{1}{\pi}m(B).
$$
\end{prop} 

\begin{proof} 
  We only have to show that for each $k$, the
  sequence 
$$\varphi_N(2\pi t/\log N)\one_{\frac{2\pi |t|\tau_N}{\log N}\le k}
$$ 
is uniformly integrable.  This is seen as follows: if $ \frac{2\pi
  |t|\tau_N}{\log N}\le k$ and if $N$ is big enough, then $|t|/\log
N\le 1/2$.  Consequently for $N$ large enough (depending on $k$), we
get 
$$
\left|\varphi_N(2\pi t/\log N)\one_{\frac{2\pi |t|\tau_N}{\log N}\le
    k}\right|\le C\exp(-|t|) + \psi_N(t),
$$
where $\Vert \psi_N\Vert_1\le CN^{-1/3}\log N$. This implies uniform
integrability of the sequence.
\end{proof}

\begin{rem} The more precise local limit theorem
$$
\frac{\log N}{2\pi}\Pr_N\[X_N\in B\]\rightarrow
\frac{1}{\pi}m(B)
$$
is in fact valid, as proved by Bruggeman~\cite{Br}.  Our methods do
not seem to lead to this result only using mod-$\varphi$ convergence.
\end{rem}

\subsection{The $\zeta$-distribution}

The $\zeta-$distributions are purely atomic, infinitely divisible,
probability distributions, denoted $\mu^\sigma$, which were considered
by Khintchine and studied in more detail in~\cite{LH}.
\par 
The measure $\mu^{\sigma}$ is defined for $\sigma>1$ as the measure
supported on the points $\{-\log(n)\mid n\in\Nr;n\ge 1\}$, such that
$$
\mu^\sigma(-\log n)=\frac{n^{-\sigma}}{\zeta(\sigma)}
$$
for $n\geq 1$. Its characteristic function is then given by
$$
\varphi^\sigma(t)=\sum_{n\ge
  1}\frac{n^{-\sigma}}{\zeta(\sigma)}e^{it(-\log
  n)}=\frac{\zeta(\sigma+it)}{\zeta(\sigma)}.
$$
\par
The limit of interest here is when $\sigma\downarrow 1$. Since the
zeta function can be written
$$
\zeta(s)=\frac{\zeta^*(s)}{s-1},
$$
where $\zeta^*(s)$ defines an entire function of $s\in\Cr$ (i.e., the
zeta function has only a simple pole with residue $1$ at $s=1$), the
behavior of $\varphi^{\sigma}(t)$ is easy to understand, namely
$$
\varphi^\sigma(t)=\frac{\zeta(\sigma+it)}{\zeta(\sigma)}
=\frac{1}{1+\frac{it}{\sigma-1}}\frac{\zeta^*(\sigma+it)}
{\zeta^*(\sigma)}.
$$
\par
Thus, if $X^{\sigma}$ are random variables with law $\mu^{\sigma}$, we
see that $(\sigma-1)X^{\sigma}$ converges in law to a ``negative''
exponential distribution supported on $]-\infty,0]$ with density
$e^x$. The characteristic is not integrable, hence we can not apply
our results. To work around this, we consider independent copies
$X_1^\sigma, X^\sigma_2$ of random variables having the law
$\mu^\sigma$, and define
$$
Y^\sigma=X_1^\sigma-X_2^\sigma.
$$ 
\par
These random variables have characteristic function given by
$$
|\varphi^\sigma(t)|^2=
\frac{|\zeta(\sigma+it)|^2}{\zeta(\sigma)^2}
=\frac{1}{1+\frac{t^2}{(\sigma-1)^2}}\frac{|\zeta^*(\sigma+it)|^2}
{\zeta^*(\sigma)^2}
$$
and hence $(\sigma-1)Y^{\sigma}$ converges in law, as
$\sigma\downarrow 1$, to a double exponential (or Laplace)
distribution, with characteristic function
$\varphi(t)=\frac{1}{1+t^2}$ and density $\frac{1}{2}\exp(-|x|)$. Thus
conditions {\bf H1} and {\bf H2} are now satisfied (in the version of
a continuous limit $\sigma\downarrow 1$). Moreover, if $1<\sigma\leq
2$ and $t$ ranges over the set where $|(\sigma-1)t|\leq k$, for $k>0$
fixed, the values of
$$
\frac{|\zeta^*(\sigma+it(\sigma-1))|^2}
{\zeta^*(\sigma)^2}
$$
vary in a bounded set. This shows that {\bf H3} also holds, and we can
apply Theorem~\ref{th-local-limit}; it follows that 
$$
\lim_{\sigma\downarrow 1}\frac{1}{\sigma-1}\Pr[a<Y^\sigma<b]=
\frac{1}{2}(b-a).
$$
for all $-\infty<a<b<+\infty$. We can make this limit explicit:
indeed, $Y^\sigma$ takes values of the form
$\log(k)-\log(n)=\log(k/n)$ where $k,n\ge 1$.  The probability that
$Y^{\sigma}=\log(k/n)$ for $(k,n)=1$ (i.e. $k$ and $n$ are coprime) is
easily seen to be
$$
\Pr[Y^\sigma=\log(k/n)]= \sum_{m\ge
  1}\frac{(km)^{-\sigma}(nm)^{-\sigma}}{\zeta(\sigma)^2}=
k^{-\sigma}n^{-\sigma}\frac{\zeta(2\sigma)}{\zeta(\sigma)^2}.
$$
\par
Hence the limit becomes, for $0<\alpha<\beta$, the formula
$$
\frac{\zeta(2\sigma)}{(\sigma-1)\zeta(\sigma)^2}
\sum_{\stacksum{(k,n)=1}{\alpha<k/n<\beta}}
k^{-\sigma}n^{-\sigma}
\rightarrow
\frac{1}{2}\log\(\frac{\beta}{\alpha}\)
$$
as $\sigma\downarrow 1$, which is equivalent (since $\zeta(2)=\pi^2/6$
and $\zeta(\sigma)\sim (\sigma-1)$) to
$$
(\sigma-1)\sum_{\stacksum{(k,n)=1}{\alpha<k/n<\beta}}
k^{-\sigma}n^{-\sigma}\rightarrow
\frac{3}{\pi^2}\log\(\frac{\beta}{\alpha}\).
$$
\par
We could not find any reference to this statement, so it might be
new (although it could certainly be proved with more traditional
methods.)

\subsection{Squarefree integers}

This section is motivated by a recent paper of Cellarosi and
Sinai~\cite{CS}, who discuss a natural probabilistic model of random
squarefree integers. As we will see, some of its properties fall into
the framework of mod-$\varphi$ convergence, with a very non-standard
characteristic function $\varphi$.
\par
The set-up, in a slightly different notation than the one used
in~\cite{CS}, is the following.  We fix a probability space $\Omega$
that is big enough to carry independent copies of random variables
$\eta_p$, with index $p$ running over the prime numbers, with the
following distribution laws:
$$
\Pr[\eta_p=1]=
\Pr[\eta_p=-1]=
\frac{p}{(p+1)^2},\quad\quad
\Pr[\eta_p=0]=\frac{p^2+1}{(p+1)^2}.
$$
\par
We consider the random variables
$$
X_n=\sum_{p\leq p_n}{\eta_p\log p},\quad\quad Q_n=\exp(X_n)
$$
for $n\geq 1$, where $p_n$ is the $n$-th prime number.  
\par
The link with~\cite{CS} is the following: in the notation
of~\cite[Th. 1.1]{CS}, the distribution of $X_n$ is the same as that
of the difference
$$
(\zeta_n-\zeta'_n)\log p_n
$$
of two independent copies $\zeta_n$ and $\zeta_n'$ of the random
variables variables
$$
\zeta_{n}=\sum_{p\leq p_n}{\nu_p\log p}
$$ 
of~\cite[Th. 1.1]{CS}, where the $\nu_p$ are independent Bernoulli
variables with
$$
\Pr[\nu_p=0]=\frac{1}{p+1},\quad\quad
\Pr[\nu_p=1]=\frac{p}{p+1}.
$$
\par
These random variables $\nu_p$ are very natural in studying squarefree
numbers. Indeed, a simple computation shows that $\nu_p$ is the limit
in law, as $x\rightarrow +\infty$, of the Bernoulli random variables
$\nu_{p,x}$ defined by
$$
\Pr[\nu_{p,x}=1]=\frac{|\{n\leq x\,\mid\, x\text{ squarefree and
    divisible by } p\}|} {|\{n\leq x\,\mid\, x\text{ squarefree}\}|}.
$$
for fixed $p$.
\par
By definition, the support of the values of $\exp(\zeta_n)$ is the set
of squarefree integers only divisible by primes $p\leq p_n$, and for
$Q_n$, it is the set of rational numbers $x=a/b$ where $a$, $b\geq 1$
are coprime integers, both squarefree, and both divisible only by
primes $p\leq p_n$. It is natural to see them as giving probabilistic
models of these numbers. We obtain mod-$\varphi$ convergence for
$X_n$:

\begin{theo}\label{th-squarefree}
Let
$$
\varphi(t)=\exp\( -4\int_0^1\sin^2\Bigl(\frac{t v}{2}\Bigr)
\frac{dv}{v}\)
$$
for $t\in\Rr$. Then $\varphi$ is an integrable characteristic function
of a probability distribution on $\Rr$, and the sequence $(X_n)$
satisfies mod-$\varphi$ convergence with $d=1$ and
$A_n(t)=A_n^*(t)=(\log p_n)t$.
\end{theo}

The proof is quite similar in principle to arguments in~\cite{CS},
though our presentation is more in the usual style of analytic number
theory.
\par
We start with the easiest part of this statement:

\begin{lemma}\label{lm-sqfl1}  
  We have $\varphi\in L^1(\Rr)$, and in fact
\begin{equation}\label{eq-sqf-l1}
|\varphi(t)|\leq C|t|^{-2}
\end{equation}
for $|t|\geq 1$ and some constant $C\geq 0$.
\end{lemma}

\begin{rem}
  The characteristic function of the limit in law of the
  (non-symmetrised) random variables $\zeta_n$ used
  in~\cite[Th. 1.1]{CS} only decays as $t^{-1}$ when $|t|\rightarrow
  +\infty$, and hence is not integrable, which prevents us from
  applying our results directly to those variables. As we will see,
  this is quite delicate: changing the constant $4$ to a constant $<2$
  would lead to a failure of this property.
\par
Below, we will see that Theorem~\ref{th-local-limit} \emph{is not
  valid} for the variables $(\log p_n)\zeta_n$.
\end{rem}

\begin{proof} Integration by parts gives that 
$$
\int_0^t \frac{\sin^2
  x}{x}\,dx=\frac{1}{2}\log t +b(t)
$$ 
where $b(t)$ tends to a constant for $t\rightarrow \infty$.  From here
we deduce that
$$
4\int_0^1\sin^2\Bigl(\frac{t v}{2}\Bigr) \frac{dv}{v} =2\log |t| +
c(t)
$$ 
where $c(t)$ remains bounded.  As a result we get~(\ref{eq-sqf-l1}),
which proves that $\varphi\in L^1$ since the function is continuous.
(Alternatively, one can check that
\begin{equation}\label{eq-ci}
\varphi(t)=\exp(-2\gamma-2\log|t|+2\mathrm{Ci}(t))
\end{equation}
where $\gamma$ is the Euler constant and $\mathrm{Ci}(t)$ is the
cosine integral function, and use the properties of the latter.)  
\end{proof}


\begin{proof}[Proof of Theorem~\ref{th-squarefree}]
Let 
$$
Y_n=\frac{1}{\log p_n}X_n
$$
and let $\psi_n$ be the characteristic function of $Y_n$, which we
proceed to compute.
\par
With $x=x_n=p_n$, we have first
\begin{align}
\psi_n(t)
=\Er\[\exp(it Y_n)\]
&=\prod_{p\leq x}\Er\[\exp\(i \frac{\log p}{\log x}t \eta_p\)\]
\nonumber\\
&=\prod_{p\leq x}\(\frac{p^2+1}{(p+1)^2} + 
\frac{2p}{(p+1)^2}\cos\(t\frac{\log p}{\log x}\)\)
\nonumber\\
&=\prod_{p\leq x}\(1-\frac{2p}{(p+1)^2}
\(1-\cos\(t\frac{\log p}{\log x}\)\)\)
\nonumber\\
&=\prod_{p\leq x}\(1 - \frac{4p}{(p+1)^2}
\sin^2\(  \frac{t}{2}\frac{\log p}{\log x}\)\)
\nonumber\\
&=\exp\(\sum_{p\leq x}  \log\(  1 - \frac{4p}{(p+1)^2}
\sin^2\(  \frac{t}{2}\frac{\log p}{\log x}\) \)\)
\label{eq-psin}
\end{align}
for all $t\in\Rr$. Now we assume $t\not=0$ (since for $t=0$, the
values are always $1$). We first show pointwise, locally uniform,
convergence.
\par
The idea to see the limit emerge in the sum over $p$ is quite
simple. First of all, we can expand the logarithm in Taylor
series. We have
$$
\lim_{x\rightarrow +\infty}\sum_{k\geq 2}\sum_{p\leq x}p^{-k}\Bigl|
\sin^2\Bigl(\frac{t\log p}{2\log x}\Bigr)
\Bigr|
=0,
$$
for $t$ in a bounded set, by dominated convergence. This allows us to
restrict our attention to
\begin{equation}\label{eq-suffice}
-4\sum_{p\leq x}p^{-1}\sin^2\Bigl(\frac{t\log p}{2\log x}\Bigr)
\end{equation}
(we also used the fact that $4p/(p+1)^2$ is equal to $4/p$ up to terms
of order $p^{-2}$.)  Now, for $p\leq y$, where $y\leq x^{1/|t|}$ is a
further parameter (assuming, as we can, that this is $\geq 2$), we
have also
$$
\Bigl|\sum_{p\leq y}
p^{-1}\sin^2\Bigl(\frac{t\log p}{2\log x}\Bigr)
\Bigr|
\leq
\Bigl(\frac{t}{2\log x}\Bigr)^2\sum_{p\leq y}
p^{-1}(\log p)^2
\ll t^2\frac{(\log y)^2}{(\log x)^2}.
$$
\par
Thus, for if we select $y=y(x)\leq x^{1/|t|}$ tending to infinity
slowly enough that $\log y=o(\log x)$, this also converges to $0$ as
$x\rightarrow +\infty$, and what remains is
$$
-4\sum_{y(x)\leq p\leq x}p^{-1}\sin^2\Bigl(\frac{t\log p}{2\log x}\Bigr).
$$
\par
We can now perform ``back-and-forth'' summation by parts using the
Prime Number Theorem to see that this is
$$
-4\int_{y(x)}^{x} u^{-1}\sin^2\Bigl(\frac{t\log u}{2\log
  x}\Bigr)\frac{du}{\log u}+o(1)
$$
as $x\rightarrow +\infty$ (apply Lemma~\ref{lm-pnt} below with $B=2$
and with the function
$$
f(u)=\frac{1}{u}\sin^2\Bigl(\frac{t\log u}{2\log x}\Bigr)
$$
with
$$
f'(u)=-\frac{1}{u^2}\sin^2\Bigl(\frac{t\log u}{2\log x}\Bigr)
+\frac{t}{u^2(\log x)}\sin\Bigl(\frac{t\log u}{2\log x}\Bigr)
\cos\Bigl(\frac{t\log u}{2\log x}\Bigr),
$$
which satisfies
\begin{equation}\label{eq-bounds}
|f(u)|\leq u^{-1},\quad\quad
|f'(u)|\leq \Bigl(1+\frac{t}{\log x}\Bigr)u^{-2}\ ;
\end{equation}
the integral error term in Lemma~\ref{lm-pnt} is then dominated by the
tail beyond $y(x)$ of the convergent integral
$$
\int_2^{+\infty}{\frac{du}{u(\log u)^2}}, 
$$
and the result follows). Performing the change of variable
$$
v=\frac{\log u}{\log x},
$$
we get the integral
$$
-4\int_{\log(y(x))/(\log x)}^{1} \sin^2\Bigl(\frac{tv}{2}\Bigr)\frac{dv}{v},
$$
which converges to $\varphi(t)$ as $x\rightarrow +\infty$.
\par
To conclude the proof of Theorem~\ref{th-squarefree}, we will prove
the following inequality, which guarantees the uniform integrability
condition {\bf H3}: for any $k\geq 1$ and $t$, $n$ with $|t|\leq
k(\log x)=k(\log p_n)$, we have
\begin{equation}\label{eq-sqf-unif}
\psi_n(t)\ll |\varphi(t)|\exp(C\log \log 3|t|)
\end{equation}
which gives the desired result since we know from~(\ref{eq-sqf-l1})
that $\varphi$ decays like $|t|^{-2}$ at infinity.
\par
We can assume that $|t|\geq 2$. Now we start with the
expression~(\ref{eq-psin}) again and proceed to deal with the sum over
$p\leq x$ in the exponential using roughly the same steps as
before. To begin with, we may again estimate the
sum~(\ref{eq-suffice}) only, since the contribution of the others
terms is bounded uniformly in $t$ and $x$:
$$
\Bigl|\sum_{k\geq 2}\sum_{p\leq x}p^{-k}
\sin^2\Bigl(\frac{t\log p}{2\log x}\Bigr)
\Bigr|\leq \sum_{k\geq 2}\sum_{p}p^{-k},
$$
which is a convergent series. After exponentiation, these terms lead
to a fixed multiplicative factor, which is fine for our
target~(\ref{eq-sqf-unif}).
\par
We next deal with the small primes in~(\ref{eq-suffice}); since
$|t|\leq k\log x$, the sine term may not lead to any decay, but we
still can bound trivially
$$
\Bigl|\sum_{p\leq y}p^{-1}\sin^2\Bigl(\frac{t\log p}{2\log x}\Bigr)
\Bigr| \leq \sum_{p\leq y}{p^{-1}}\ll \log\log y
$$
for any $y\leq x$ (by a standard estimate). We select $y=|t|\geq 2$,
and this becomes a factor of the type
$$
\exp(C\log\log t)
$$
(after exponentiating), which is consistent with~(\ref{eq-sqf-unif}).
\par
We now apply Lemma~\ref{lm-pnt} again, writing more carefully the
resulting estimate, namely
\begin{align*}
-4\sum_{y<p\leq x}p^{-1}\sin^2\Bigl(\frac{t\log p}{2\log x}\Bigr)
&=
-4\int_{y}^{x} u^{-1}\sin^2\Bigl(\frac{t\log u}{2\log
  x}\Bigr)\frac{du}{\log u}+O\Bigl(\frac{1+k}{(\log y)^2}\Bigr)
\\
&=-4\int_{\log y/\log x}^1{v^{-1}\sin^2\Bigl(\frac{tv}{2}\Bigr)dv}+
O\Bigl(\frac{1+k}{(\log y)^2}\Bigr)
\end{align*}
(using the bound~(\ref{eq-bounds})), with an absolute implied
constant. The remainder here is again fine, since $y=|t|\geq 2$ by
assumption.
\par
Now, to conclude, we need only estimate the missing part of the target
integral (which runs from $0$ to $1$) in this expression, namely
$$
\int_0^{\log y/\log x}{v^{-1}\sin^2\Bigl(\frac{tv}{2}\Bigr)dv}.
$$
\par
We write
$$
\int_0^{\log y/\log x}{v^{-1}\sin^2\Bigl(\frac{tv}{2}\Bigr)dv}
=\int_{0}^{|t|^{-1}}{(\cdots)}+
\int_{|t|^{-1}}^{\log y/\log x}{(\cdots)}
$$
where the first terms is bounded by
$$
(t/2)^2\int_0^{|t|^{-1}}{vdv}\leq 1,
$$
and the second by
$$
\int_{|t|^{-1}}^{\log y/\log x}{v^{-1}dv}=\log\Bigl(\frac{|t|\log
  y}{\log x}\Bigr)
\leq \log(k\log y)=\log(k\log |t|).
$$
\par
Putting the inequalities together, we have proved~(\ref{eq-sqf-unif}),
and hence Theorem~\ref{th-squarefree}.
\end{proof}

Here is the standard lemma from prime number theory that we used
above, which expresses the fact that for primes sufficiently large,
the heuristic -- due to Gauss -- that primes behave like positive
numbers with the measure $du/(\log u)$ can be applied confidently in
many cases.

\begin{lemma}\label{lm-pnt}
  Let $y\geq 2$ and let $f$ be a smooth function defined on
  $[y,+\infty[$. Then for any $A\geq 1$, we have
$$
\sum_{y\leq p\leq x}{f(p)}= \int_y^x{f(u)\frac{du}{\log u}}+O\Bigl(
\frac{x|f(x)|}{(\log x)^A}+\frac{y|f(y)|}
{(\log y)^A}+ \int_y^x{|f'(u)|\frac{udu}{(\log
    u)^A}} \Bigr)
$$
where the sum is over primes and the implied constant depends only on
$A$.
\end{lemma}

We give the proof for completeness.

\begin{proof}
  We use summation by parts and the Prime Number Theorem, which is the
  case $f(x)=1$, in the strong form
$$
\pi(x)=\int_2^x{\frac{du}{\log u}}+O\Bigl(\frac{x}{(\log x)^A}\Bigr)
$$
for $x\geq 2$ and any $A\geq 1$, with an implied constant depending
only on $A$ (this is a consequence of the error term in the Prime
Number Theorem due to de la Vall\'ee Poussin, see
e.g.\cite[Cor. 5.29]{IK}); this leads to
$$
\sum_{y\leq p\leq x}{f(p)}=f(x)\pi(x)-f(y)\pi(y)
-\int_y^x{f'(u)\pi(u)du},
$$
and after inserting the above asymptotic formula for $\pi(x)$ and
$\pi(u)$, we can revert the integration by parts to recover the main
term, while the error terms lead to the result.
\end{proof}

We now derive arithmetic consequences of Theorem~\ref{th-squarefree}.
Applying Theorem 2, we get
\begin{equation}\label{eq-sqf-appl}
\lim_{n\rightarrow +\infty}
(\log p_n)\Pr\[X_n\in ]a,b[\]=(b-a)\eta
\end{equation}
where 
$$
\eta=\frac{1}{2\pi}\int_{-\infty}^{+\infty}
\exp\(-4\int_0^1\frac{\sin^2\frac{tv}{2}}{v}\,dv\)\,dt=
\frac{1}{2\pi}\int_{-\infty}^{+\infty}
e^{2(\mathrm{Ci}(t)-\gamma)}\frac{dt}{|t|^2},
$$ 
the last expression coming from~(\ref{eq-ci}). Using the relation
between $\varphi$ and the Dickman-de Bruijn function $\rho$, namely
$$
\varphi(t)=\psi(t)\psi(-t)
$$
where $\psi(t)$ is the Fourier transform of $e^{-\gamma}\rho(u)$ (this
follows from~\cite[Th. 1.1, p. 5]{CS}), one gets
$$
\eta=e^{-2\gamma}\int_{\Rr}{\rho(u)^2du}=
0.454867\ldots
$$
(the numerical computation was done using \textsf{Sage}).
\par
This arithmetic application could certainly be proved with more
traditional methods of analytic number theory, when expressed
concretely as giving the asymptotic behavior as $n\rightarrow +\infty$
of
$$
\sum_{\alpha<r/s<\beta}{\Pr\[X_n=\frac{r}{s}\]},
$$
but it is nevertheless a good illustration of the general
probabilistic framework of mod-$\varphi$ convergence with an unusual
characteristic function.
\par
Although our theorem does not apply for the random model of~\cite{CS}
itself, it is quite easy to understand the behavior of the
corresponding probabilities in that case. Indeed, denoting
$$
Y_n=\exp((\log p_n)\zeta_n),
$$
which takes squarefree values, we have 
$$
\Pr[Y_n<e^a]=
\frac{1}{Z_x}\sum_{\stacksum{k<e^a}{p\mid k\Rightarrow p\leq
    x}}{\frac{\mu^2(k)}{k}}
$$
for any fixed $a\in\Rr$, where $x=p_n$, $\mu^2(k)$ is the indicator
function of squarefree integers and $Z_x$ is the normalizing factor
given by
$$
Z_x=\prod_{p\leq x}{(1+p^{-1})}.
$$
\par
For $x$ large enough and $a$ fixed, the second condition is vacuous,
and hence this is
$$
\frac{1}{Z_x}\sum_{k<e^a}{\frac{\mu^2(k)}{k}}.
$$
\par
As observed in~\cite[(3)]{CS}, we have $Z_x\sim
e^{\gamma}\zeta(2)^{-1}\log x$, and hence we get
$$
\lim_{n\rightarrow +\infty}{
(\log p_n)\Pr[Y_n<e^a]
}=
\zeta(2)e^{-\gamma}\sum_{k<e^a}{\frac{\mu^2(k)}{k}}.
$$
\par
When $a$ is large, this is equivalent to $e^{-\gamma}a$ (another easy
fact of analytic number theory), which corresponds to the local limit
theorem like~(\ref{eq-sqf-appl}), but we see that for fixed $a$, there
is a discrepancy.
\par
There is one last interesting feature of this model: the analogue of
Theorem~\ref{th-squarefree} for polynomials over finite fields
\emph{does not hold}, despite the many similarities that exist between
integers and such polynomials (see, e.g.,~\cite{KN2} for instances of
these similarities in related probabilistic contexts.)
\par
Precisely, let $q>1$ be a power of a prime number and $\Fr_q$ a finite
field with $q$ elements. For irreducible monic polynomials
$\pi\in\Fr_q[X]$, we suppose given independent random variables
$\eta_{\pi}$, $\eta'_{\pi}$ such that by
$$
\Pr[\eta_{\pi}=\pm 1]=\Pr[\eta'_{\pi}=\pm 1]=
\frac{|\pi|}{(|\pi|+1)^2},\quad\quad
\Pr[\eta_{\pi}=0]=\Pr[\eta'_{\pi}=0]=\frac{|\pi|^2+1}{(|\pi|+1)^2}
$$
where $|\pi|=q^{\deg(\pi)}$. Then for $n\geq 1$, let $\hat{X}_n$ be
the random variable
$$
\sum_{\deg(\pi)\leq n}{(\deg\pi)(\eta_{\pi}-\eta'_{\pi})},
$$
where the sum runs over all irreducible monic polynomials of degree at
most $n$. Then we claim that {\bf H1}, {\bf H2} hold for $\hat{X}_n$,
with the same characteristic function $\varphi(t)$ as in
Theorem~\ref{th-squarefree}, and $A_nt=nt$, but there is no
mod$-\varphi$ convergence.
\par
This last part at least is immediate: {\bf H3} fails by contraposition
because the local limit theorem for
$$
\lim_{n\rightarrow +\infty} n\Pr[a<\hat{X}_n<b]
$$
is not valid! Indeed, $\hat{X}_n$ is now real-valued, and if
$]a,b[\cap \Zr=\emptyset$, the probability above is always $0$,
whereas the expected limit $(b-a)\eta$ is not.
\par
We now check {\bf H2} in this case. Arguing as in the beginning of the
proof of Theorem~\ref{th-squarefree}, we get
$$
\Er[e^{it\hat{X}_n/n}]=
\prod_{\deg(\pi)\leq n}{
\(
1-\frac{4|\pi|}{(|\pi|+1)^2}\sin^2\(\frac{\deg(\pi)t}{2n}\)
\)
}.
$$
\par
Expanding the logarithm once more, we see that it is enough to prove
that (locally uniformly in $t$) we have
$$
\lim_{n\rightarrow +\infty}
\exp\(
-4\sum_{\deg(\pi)\leq n}{\frac{1}{|\pi|}\sin^2\(\frac{\deg(\pi)t}{2n}\)}
\)=\varphi(t).
$$
\par
We arrange the sum according to the degree of $\pi$, obtaining
$$
\sum_{\deg(\pi)\leq n}{\frac{1}{|\pi|}\sin^2\(\frac{\deg(\pi)t}{2n}\)}
=\sum_{j=1}^n{\frac{1}{q^j}\sin^2\(\frac{jt}{2n}\)\Pi_q(j)}
$$
where $\Pi_q(j)$ is the number of monic irreducible polynomials of
degree $j$ in $\Fr_q[X]$. The well-known elementary formula of Gauss
and Dedekind for $\Pi_q(j)$ shows that
$$
\Pi_q(j)=\frac{q^j}{j}+O(q^{j/2})
$$
for $q$ fixed and $j\geq 1$, and hence we can write the sum as
$$
\sum_{\deg(\pi)\leq n}{\frac{1}{|\pi|}\sin^2\(\frac{\deg(\pi)t}{2n}\)}
=\sum_{j=1}^n{\frac{1}{j}\sin^2\(\frac{jt}{2n}\)}+O
\(\sum_{j=1}^n{q^{-j/2}\sin^2\(\frac{tj}{2n}\)}\).
$$
\par
As $n$ goes to infinity, the second term converges to $0$ by the
dominated convergence theorem, while the first is a Riemann sum (with
steps $1/n$) for the integral
$$
\int_0^1{\sin^2\(\frac{tv}{2}\)\frac{dv}{v}},
$$
and hence we obtain the desired limit. (This is somewhat similar
to~\cite[Prop. 4.6]{abt}.)

\begin{rem}
  A more purely probabilistic example of the same phenomenon arises as
  follows: define
$$
\tilde{X}_n=\sum_{j=1}^n{(D_j-E_j)}
$$
where $(D_j, E_j)$ are globally independent random variables with
distribution
$$
\Pr[E_j=j]=\Pr[D_j=j]=\frac{1}{j},\quad\quad 
\Pr[E_j=0]=\Pr[D_j=0]=1-\frac{1}{j}.
$$
\par
Then the sequence $(\tilde{X}_n)$ also satisfies {\bf H1} and {\bf H2}
for the same characteristic function $\varphi(t)$ (by very similar
arguments), and does not satisfy {\bf H3} since $\tilde{X}_n$ is
integral-valued.
\end{rem}

\subsection{Random Matrices}
\label{sec-random-matrices}

Some of the first examples of mod-Gaussian convergence are related to
the ``ensembles'' of random matrices corresponding to families of
compact Lie groups, as follows from the work of Keating and
Snaith~\cite{KS1}, \cite{KS2}. Using this, and our main result, we can
deduce quickly some local limit theorems for values of the
characteristic polynomials of such random matrices.
\par
We consider the three standard families of compact matrix groups,
which we will denote generically by $\Gb_n$, where $\Gb$ is either $U$
(unitary matrices of size $n$), $USp$ (symplectic matrices of size
$2n$) or $SO$ (orthogonal matrices of determinant $1$ and
size\footnote{\ The odd case could be treated similarly.}  $2n$). In
each case, we consider $\Gb_n$ as a probability space by putting the
Haar measure $\mu_n$ on $\Gb_n$, normalized so that
$\mu_n(\Gb_n)=1$. The relevant random variables $(X_n)$ are defined as
suitably centered values of the characteristic polynomial
$\det(T-g_n)$ where $g_n$ is a $\Gb_n$-valued random variable which is
$\mu_n$-distributed. Precisely, define
$$
\alpha_n=
\left\{
\begin{array}{cc}
0&\text{ if } \Gb=U, \\
\frac{1}{2}\log (\pi n/2)&\text{ if } \Gb=USp,\\
\frac{1}{2}\log (8\pi/n)&\text{ if } \Gb=SO,
\end{array}
\right.
$$
and consider $X_n=\log\det(1-g_n)-\alpha_n$; this is real-valued
except for $\Gb=U$, in which case the determination of the logarithm
is obtained from the standard Taylor series at $z=1$.
\par
Now define the linear maps
$$
A_n(t)=
\left\{
\begin{array}{cc}
\(\frac{\log n}{2}\)^{1/2}(t_1,t_2)&\text{ if } \Gb=U, \\
\(\log\(\frac{n}{2}\)\)^{1/2}t&\text{ otherwise,}
\end{array}
\right.
$$
and their inverses $\Sigma_n$ (these are diagonal so $A_n^*=A_n$,
$\Sigma_n^*=\Sigma_n$). 
\par
Finally, let $\varphi$ be the characteristic function of a standard
complex (if $\Gb=U$) or real gaussian random variable (if $\Gb=USp$ or
$SO$); in particular {\bf H1} is true. It follows from the work of
Keating and Snaith that in each case $\varphi_n(\Sigma_nt)$ converges
continuously to $\varphi(t)$, i.e., that ${\bf H2}$ holds. In fact, in
each case, there is a continuous (in fact, analytic) limiting function
$\Phi_{\Gb}(t)$ such that
$$
\varphi_n(t)=\varphi(A_n^*t)\Phi_{\Gb}(t)(1+o(1))
$$
for any fixed $t$, as $n$ goes to infinity. These are given by
$$
\Phi_{\Gb}(t)=
\left\{
\begin{array}{cc}
\frac{G(1+\frac{it_1-t_2}{2})
G(1+\frac{it_1+t_2}{2})}{G(1+it_1)}&\text{ if } \Gb=U, \\
\frac{G(3/2)}{G(3/2+it)}&\text{ if } \Gb=USp,\\
\frac{G(1/2)}{G(1/2+it)}&\text{ if } \Gb=SO,
\end{array}
\right.
$$
in terms of the Barnes $G$-function. Detailed proofs can be found
in~\cite[\S 3, Prop. 12, Prop. 15]{KN}, and from the latter arguments,
one obtains uniform estimates
$$
|\varphi_n(t)|\leq C|\Phi_{\Gb}(t)\varphi(A_n^*t)|
$$
for all $t$ such that $|t|\leq n^{1/6}$, where $C$ is an absolute
constant. This immediately gives the uniform integrability for
$\varphi_n(\Sigma_n^*t)\one_{|\Sigma_nt|\leq k}$ since $|\Sigma_n^*
t|$ is only of logarithmic size with respect to $n$. In other words,
we have checked {\bf H3}, and hence there is mod-$\varphi$
convergence.
\par
Consequently, applying Theorem~\ref{th-local-limit}, we derive the
local limit theorems (already found in~\cite{KN}):

\begin{theo}
  For $\Gb=U$, $USp$ or $SO$, for any bounded Jordan-measurable set
  $B\subset \Rr$ or $\Cr$, the latter only for $\Gb=U$, we have
$$
\lim_{n\rightarrow +\infty} |\det(A_n)|\quad
{\mu_n(g\in\Gb_n\,\mid\, 
\log \det(1-g)-\alpha_n\in B)}=\frac{m(B)}{(2\pi)^{d/2}}
$$
with $d=2$ for $\Gb=U$ and $d=1$ otherwise.
\end{theo}

Theorem~\ref{th-unitary}, stated in the introduction, is the special
case $\Gb=U$, enhanced by applying Proposition~\ref{pr-shift-mean}. 
\par
As in~\cite[\S 4]{KN}, one can derive arithmetic consequences of these
local limit theorems, involving families of $L$-functions over finite
fields, by appealing to the work of Katz and Sarnak. The interested
readers should have no difficulty checking this using the detailed
results and references in~\cite{KN}.
\par
Instead, we discuss briefly a rather more exotic type of random
matrices, motivated by the recent results in~\cite{kst} concerning
certain averages of $L$-functions of Siegel modular
forms. In~\cite[Rem. 1.3]{kst}, the following model is suggested: let
$G_n=SO_{2n}(\Rr)$, with Haar measure $\mu_n$, and consider the
measure
$$
\nu_n(g)=\frac{1}{2}\det(1-g)d\mu_n(g)
$$
on $G_n$. The density $\det(1-g)$ is non-negative on $G_n$ (because
eigenvalues of a matrix in $SO(2n,\Rr)$ come in pairs $e^{i\theta}$,
$e^{-i\theta}$, and $(1-e^{i\theta})(1-e^{-i\theta})\geq 0$); the fact
that this is a probability measure will be explained below. In
probabilistic terms, this is the ``size-biased'' version of $\mu_n$.

\begin{theo}
  Let $X_n=\log\det(1-\tilde{g}_n)-\frac{1}{2}\log( 32\pi
  n)$, where $\tilde{g}_n$ is a $G_n$-valued random variable
  distributed according to $\nu_n$. Let $\varphi$ be the
  characteristic function of a standard real gaussian. Then we have
  mod-$\varphi$ convergence with $A_nt=(\log \frac{n}{2})^{1/2}t$, and
  in particular
$$
\lim_{n\rightarrow +\infty} \sqrt{\log\frac{n}{2}}\quad
{\nu_n\(g\in\Gb_n\,\mid\, \log \det(1-g)-\frac{1}{2}\log
  (32\pi n)\in B\)}=\frac{m(B)} {\sqrt{2\pi}}.
$$
\end{theo}

\begin{proof}
  The characteristic function of $Y_n=\log\det(1-\tilde{g}_n)$ is half
  of the value at $s=1+it$ of the Laplace transform
  $\Er[e^{s\log\det(1-g_n)}]$, where $g_n$ is Haar-distributed. The latter
  is computed for all complex $s$ in~\cite[(56)]{KS1}, and we get
$$
2\Er[e^{itY_n}]=
\Er[e^{(1+it)\log \det(1-g_n)}]=2^{2n(1+it)}
\prod_{1\leq j\leq n}
{
\frac{\Gamma(j+n-1)\Gamma(j+it+1/2)}{\Gamma(j-1/2)\Gamma(j+it+n)}
}.
$$
\par
At this point, the reader may check easily (by recurrence on $n$ if
needed) that this gives the right values $\Er[e^{itY_n}]=1$ for $t=0$,
confirming the normalizing factor $1/2$ used in the definition of
$\nu_n$.
\par
To go further, we transform the right-hand side into values of the
Barnes function $G(z)$, as in~\cite[\S 4.3]{KN}, to get
$$
2\Er[e^{itY_n}]=2^{2n(1+it)}\frac{G(1/2)}{G(3/2+it)}
\frac{G(2n)G(n+3/2+it)G(1+it+n)}{G(n)G(n+1/2)G(2n+1+it)}.
$$
\par
Applying $\Gamma(z)G(z)=G(z+1)$, we transform this into
\begin{align*}
2\Er[e^{itY_n}]=2^{2n(1+it)}\frac{G(1/2)}{G(3/2+it)}&
\frac{\Gamma(it+n)\Gamma(it+n+1/2)}{\Gamma(it+2n)}\\
&\quad\quad\quad\times\frac{G(2n)G(n+it)G(1/2+it+n)}{G(n)G(n+1/2)G(2n+it)}
\end{align*}
and the last ratio of Barnes functions (together with the factor
$2^{2nit}$) is exactly the one handled in~\cite[Prop. 17,
(4)]{KN}. With the asymptotic formula that follows, the Legendre
duplication formula and $\Gamma(1/2)=\sqrt{\pi}$, we deduce
$$
2\Er[e^{itY_n}]=
2\frac{G(3/2)}{G(3/2+it)}
\frac{\Gamma(2it+2n)}{\Gamma(it+2n)}
\(\frac{n}{2}\)^{-t^2/2}\(\frac{8\pi}{n}\)^{it/2}(1+o(1))
$$
uniformly for $|t|\leq n^{1/6}$. Since
$$
\frac{\Gamma(2it+2n)}{\Gamma(it+2n)}
=(2n)^{it}(1+o(1))
$$
in this range, we get
$$
\Er[e^{itY_n}]= \frac{G(3/2)}{G(3/2+it)} \varphi(A_n^*t)(32\pi
n)^{it/2}(1+o(1)),
$$
uniformly for $|t|\leq n^{1/6}$, and the result follows.
\end{proof}

The most obvious feature of this exotic model of orthogonal matrices
is the ``shift'' of the average; whereas, for Haar-distributed $g\in
SO_{2n}(\Rr)$, the value $\log \det(1-g)$ is typically small (mean
about $\log\sqrt{8\pi/n}$), it becomes typically large (mean $\log
\sqrt{32\pi n}$, of similar order of magnitude as the mean for a
symplectic matrix of the same size) when $g$ is considered to be
distributed according to $\nu_n$. This is consistent with the
discussion in~\cite[Rem. 1.3]{kst}, especially since the ``limiting
function'' that appears here is $\Phi_{USp}$.

\subsection{Stochastic model of the Riemann zeta function}

The following ``naive'' model of the Riemann zeta function on the
critical line is surprisingly helpful.  The basic ingredient is a
sequence of iid variables $Y_p\colon \Omega \rightarrow \Tr$ where
$\Tr$ is the unit circle in $\Cr$ and the variables $Y_p$ are
uniformly distributed over $\Tr$. For notational ease the sequence is
ordered by the prime numbers. In what follows $p$ will always denote a
prime number. The random variables we consider are constructed as
follows.  First we take finite products $Z_n=\prod_{p\le
  n}\(1-\frac{Y_p}{\sqrt{p}}\)$.  If we replace the factors $Y_p$ by
$\exp(ipt)$, then the product appears in the study of the Riemann
$\zeta$-function. An easy application of Weyl's lemma on uniform
distributions shows that $\(\exp(ipt)\)_{p\le n}$ defined on $[0,T]$
(with normalised Lebesgue measure) tend (as $T\rightarrow \infty$) to
$(Y_p)_{p\le n}$. The random variables $X_n$ are then defined as minus
the logarithm of $Z_n$, (taken along its principal branch defined as
$\log(1)=0$).  So 
$$
X_n=-\sum_{p\le
  n}\log\(1-\frac{Y_p}{\sqrt{p}}\)=\sum_{p\le
  n}\sum_k\frac{1}{k}\(\frac{Y_p}{\sqrt{p}}\)^k.
$$
\par
These sums clearly converge. Because of this explicit form we can
calculate the characteristic functions. The calculations are done in
\cite[\S 3, Ex. 2]{KN} and this yields the following.
$$
\varphi_n(t)=\Er[\exp(it\cdot X_n)]= \prod_{p\le
  n}{}_2F_1\(\frac{1}{2}(it_1+t_2),
\frac{1}{2}(it_1-t_2);1;\frac{1}{p}\),
$$
where $t=(t_1,t_2)\in \Rr^2$, $t\cdot x=t_1 x_1+t_2x_2$ is the inner
product in $\Rr^2$ and ${}_2F_1$ denotes the Gauss hypergeometric
function. Straightforward estimates (see~\cite{KN} for details) then
give
\begin{enumerate}
\item $|\varphi_n(t)| \le c(t) \exp(-\frac{1}{16}(\log\log n) |t|^2)$,
  where $c$ is a non-decreasing function (in fact one can take a
  constant);
\item $\varphi_n\(\sqrt{\frac{2}{\log\log n}}t\)\rightarrow
  \exp(-\frac{1}{2}|t|^2)$.
\end{enumerate}

The conditions of Theorem~\ref{th-local-limit} are fulfilled and hence
we have
$$
\frac{\log\log n}{2}\Pr[X_n\in B]\rightarrow \frac{1}{2\pi} m(B).
$$
for any bounded Jordan measurable set $B\subset \Cr$.

\subsection{The Riemann zeta function on the critical line}

The results in this section are conjectural, but they are of interest
to number theorists.  By work of Selberg, the central limit theorem
for $\log \zeta(1/2+it)$ is known, after renormalizing by
$\sqrt{\log\log T}$, see e.g.~\cite{HNY}. This is proved by asymptotic
estimations of the moments, and there is no known bound for the
corresponding characteristic functions. Thus, we cannot currently
apply our theorems.
\par
However, Keating and Snaith (\cite{KS1}, \cite{KS2}) have proposed the
following precise conjecture (based on links with Random Matrix
Theory) concerning the characteristic function: for any
$t=(t_1,t_2)\in\Rr^2$, they expect that
$$
\frac{1}{T}\int_0^T{\exp(i t\cdot \log\zeta(1/2+iu))du}
\sim \Phi(u)\exp\Bigl(-\frac{\log\log T}{4}|u|^2\Bigr)
$$
as $T\rightarrow +\infty$, where the limiting function is the product
of the corresponding factors for unitary matrices and for the
``stochastic'' version of $\zeta$, described in the previous sections,
i.e.,
\begin{align*}
  \Phi(t_1,t_2)&=\frac{G(1+\frac{it_1-t_2}{2})
    G(1+\frac{it_1+t_2}{2})}{G(1+it_1)}
  \\
  &\times \prod_{p}{{}_2F_1(\frac{1}{2}(it_1+t_2),
    \frac{1}{2}(it_1-t_2);1;p^{-1})}
\end{align*}
(the normalization of $\log\zeta(1/2+iu)$ is obtained by continuation
of the argument from the value $0$ for $\zeta(\sigma+iu)$ when
$\sigma$ real tends to infinity, except for the countable set of $u$
which are ordinates of zeros of $\zeta$.)
\par
In \cite[Cor. 9]{KN}, it is shown that a suitable uniform version of
this conjecture implies local limit theorems for
$$
\frac{1}{T}m(u\in [0,T]\,\mid\, \log\zeta(1/2+iu)\in B)
$$
and, as a corollary, implies that the set of values of
$\zeta(1/2+iu)$, $u\in\Rr$, is \emph{dense} in $\Cr$, which is an old
and intriguing conjecture of Ramachandra.
\par
The mod-$\varphi$ framework allows us to show that a much weaker
statement than the one considered in~\cite{KN} is already sufficient
to get the same local limit theorems. Indeed, we consider the
following much statement, which of course implies Ramachandra's
conjecture, as being very likely to be true:

\begin{conj}[Quantitative density of values of $\zeta(1/2+it$)]
  \label{conj-zeta}
  For any bounded Jordan-measurable subset $B\subset \Cc$, we have
$$
\lim_{T\rightarrow+\infty} \frac{\frac{1}{2}\log\log T} {T} m(u\in
[0,T]\,\mid\, \log\zeta(1/2+iu)\in B)=\frac{m(B)}{2\pi}.
$$
\end{conj}

The point is that this follows using Theorem~\ref{th-local-limit} from
fairly weak decay estimates for the characteristic function of
$\log\zeta(1/2+it)$ (in comparison with what the Keating-Snaith
conjecture suggests). For instance, if for all $k>0$ there exists
$C_k\geq 0$ such that
\begin{equation}\label{eq-hypothesis-zeta}
\Bigl|\frac{1}{T}\int_0^T{\exp(i t\cdot 
\log\zeta(1/2+iu))du}\Bigr|
\leq \frac{C_k}{1+|t|^4(\log\log T)^2}
\end{equation}
for all $T\geq 1$ and $t$ with $|t|\leq k$, then
Conjecture~\ref{conj-zeta} is true. 
Indeed, in Theorem~\ref{th-local-limit}, we can take $\varphi$ to be
the characteristic function of a standard complex gaussian and $X_n$
to be (for some arbitrary sequence $T_n$ going to $+\infty$) a random
variable with law given by the probability distribution of
$\log\zeta(1/2+iu)$ for $u$ uniform on $[0,T_n]$. These satisfy {\bf
  H1} trivially, and {\bf H2} holds with
$$
A_n(t_1,t_2)=A_n^*(t_1,t_2)=\sqrt{\frac{1}{2}\log\log T_n}
(t_1,t_2),
$$
because of Selberg's Central Limit Theorem. The
hypothesis~(\ref{eq-hypothesis-zeta}) states that, for any $k>0$, we
have
$$
|\varphi_n(t)|\leq C_kh(A_n^*t)
$$
for $|t|\leq k$, with
$$
h(t_1,t_2)=\frac{1}{1+4|t|^4},
$$
or equivalently
$$
|\varphi(\Sigma_n^*t)|\leq C_k h(t)
$$
for $|\Sigma_n^*t|\leq k$. Since $h\in L^1(\Rr^2)$, this
gives~(\ref{eq-dominated}), and we get the conjectured statement from
the local limit theorem.
\par
The significance of this remark is the fact that, for fixed $t\not=0$,
the decay rate of the characteristic function which is required is
``only'' of order $(\log\log T)^{-2}$, which is much weaker than what
is suggested by the Keating-Snaith conjecture, and therefore might be
more accessible.

\section*{Appendix A}

We sketch here a proof of Theorem~\ref{th-approx}. Suppose the support
of $f$ is contained in $[-k+1,k-1]^d$.  Since we can construct
approximations for $f^+$ and $f^-$ separately, we can assume without
loss of generality that $f\ge 0$. 
\par
Let $\ep>0$ and let $N=\ep^{-1/(2d)}$, assuming $\ep$ small enough
that $N>k$. 
Let then $\theta$ be a continuous function on $\Rr^d$ such that
$$
\ep\le\theta\le \ep +\ep^{1/(4d)},
$$ 
and $\theta=\ep +\ep^{1/(4d)}$ on the support of $f$, while
$\theta=\ep$ outside $[-k,k]^d$. Further, let $p$ be a trigonometric
polynomial in $d$ variables, with periods $(2N,\cdots,2N)$, which
approximates $f+\theta$ uniformly on $[-N,N]^d$ up to an error
$\ep$. Clearly, we have then 
$$
p\ge f+\theta-\ep\ge f\ge 0.
$$
\par
The function $p$ is considered as a periodic function on $\Rr^d$, and
it remains non-negative of course.  The Fourier transform of $p$, in
the sense of distributions, is a finite linear combination of Dirac
measures, hence has compact support.
\par
Now we consider
$$
h(x)=\prod_j \frac{\sin^2(ax_j)}{(ax_j)^2}
$$
where $a=\frac{\sqrt{2\delta}}{k}$, for some $\delta>0$; we find that
$$
(1-\delta)^d \le h\le 1
$$
on $[-k,k]^d$. We select
$$
\delta=\frac{\theta(0)-\ep}{d(\theta(0)-\ep+\Vert f\Vert_\infty)},
$$
and then we claim that the function $g_1=ph$ satisfies $g_1\ge f$,
while $\int (g_1-f)dm$ can be made arbitrarily small if $\ep$ is small
enough.
\par
Indeed, first of all we have $g_1\ge 0$ everywhere, while on the
support of $f$ we get
\begin{align*}
  ph-f
  &\ge(1-\delta)^d(f+\theta-\ep)-f\\
  &\ge(1-\delta)^d(\theta-\ep)-(1-(1-\delta)^d)f\\
  &\ge (1-d\delta)(\theta-\ep)-d\delta\Vert f\Vert_\infty\\
  &\ge (\theta-\ep)-d\delta(\theta-\ep+\Vert f\Vert_\infty)\ge
  0,\quad\text{ by the choice of $\delta$.}
\end{align*}
\par
Next the integral $\int(g_1-f)dm$ is estimated as follows (using the
notation $C$ for non-negative constants, the value of which may change
from line to line, and $|x|_\infty=\max_{1\le j\le d}|x_j|$). On
$|x|_\infty\le k$, we have
$$
\int_{|x|_\infty\le k}(g_1-f)dm
\le \int_{|x|_\infty\le k} (f+\theta+\ep-f)dm
\le (2k)^d(2\ep +\ep^{1/(4d)}).
$$
\par
On the set $\{k<|x|_\infty\le N\}$, using the estimate on
$p$, we get
$$
\int_{k<|x|_\infty\le N}(g_1-f)dm\le \int_{|x|_\infty\le N}p\ dm\le
2\ep(2N)^d\le C \ep^{1/2}.
$$
\par
Finally, where $|x|_\infty>N$, we estimate using the bound on $h$:
\begin{align*}
  \int_{N<|x|_\infty} (g_1-f)dm&\le 
\int_{N<|x|_\infty}h(\Vert f\Vert_\infty+\theta+\ep)dm\\
  &\le d\int_{|x_1|\ge N}h(\Vert f\Vert_\infty+\theta+\ep)dm
  \le C \frac{\Vert f\Vert_\infty+\theta(x)+\ep}{Na^{d+1}}.
\end{align*}
\par
As $\ep\rightarrow 0$, the denominator $a^{d+1}N$ tends to $\infty$
since
$$
a^{d+1}N=\ep^{-1/2}\(\frac{\ep^{1/(4d)}}{(d\ep^{1/(4d)}+\Vert
  f\Vert_\infty)k^2}\)^{(d+1)/2},
$$
and therefore $\int{(g_1-f)dm}$ can be made arbitrarily small by
choosing $\ep$ small enough, as claimed.
\par
To conclude, we note that the Fourier transform of $ph$ is, up to a
constant, a convolution of the Fourier transforms of $p$ and $h$.
Since the Fourier transform of $h$ is supported on $[-2a,2a]^d$, the
support of the Fourier transform of $ph$ is therefore contained in the
sum of two compact sets, which is compact.
\par
Similarly,  using the function $f-\theta$, we construct the
approximation function $g_2\leq f$.

\end{document}